\documentclass[12pt,leqno]{article}

\usepackage{graphicx, fullpage}
\usepackage{amsmath,amssymb,amsthm,amscd, bm}
\usepackage{fancyhdr}
\usepackage[mathscr]{eucal}
\usepackage{amsfonts}
\usepackage[symbol]{footmisc}
\usepackage[T1]{fontenc}
\usepackage{xspace}
\begin{document}
\parskip=6pt

\theoremstyle{plain}
\newtheorem{prop}{Proposition}
\newtheorem{lem}[prop]{Lemma}
\newtheorem{thm}[prop]{Theorem}
\newtheorem{cor}[prop]{Corollary}
\newtheorem{defn}[prop]{Definition}
\theoremstyle{definition}
\newtheorem{example}[prop]{Example}
\theoremstyle{remark}
\newtheorem{remark}[prop]{Remark}
\numberwithin{prop}{section}
\numberwithin{equation}{section}

\newenvironment{rcases}
  {\left.\begin{aligned}}
  {\end{aligned}\right\rbrace}


\def\cal{\mathcal}
\newcommand{\cF}{\cal F}
\newcommand{\cG}{\cal G}
\newcommand{\cA}{\cal A}
\newcommand{\cB}{\cal B}
\newcommand{\cC}{\cal C}
\newcommand{\cO}{{\cal O}}
\newcommand{\cE}{{\cal E}}
\newcommand{\cH}{{\cal H}}
\newcommand{\cU}{{\cal U}}
\newcommand{\cL}{{\cal L}}
\newcommand{\cM}{{\cal M}}
\newcommand{\cD}{{\cal D}}
\newcommand{\cK}{{\cal K}}
\newcommand{\cZ}{{\cal Z}}
\newcommand{\cQ}{{\cal Q}}

\newcommand{\fQ}{\frak{Q}}

\newcommand{\bC}{\mathbb C}
\newcommand{\bP}{\mathbb P}
\newcommand{\bN}{\mathbb N}
\newcommand{\bA}{\mathbb A}
\newcommand{\bR}{\mathbb R}
\newcommand{\oL}{\overline L}
\newcommand{\oP}{\overline P}
\newcommand{\op}{\overline \partial}
\newcommand{\oQ}{\overline Q}
\newcommand{\oR}{\overline R}
\newcommand{\oS}{\overline S}
\newcommand{\oc}{\overline c}
\newcommand{\bp}{\mathbb p}
\newcommand{\oD}{\overline D}
\newcommand{\oE}{\overline E}
\newcommand{\oC}{\overline C}
\newcommand{\of}{\overline f}
\newcommand{\ou}{\overline u}
\newcommand{\oU}{\overline U}
\newcommand{\ow}{\overline w}
\newcommand{\oy}{\overline y}
\newcommand{\oz}{\overline z}
\newcommand{\oxi}{\overline \xi}

\newcommand{\hg}{\hat G}
\newcommand{\hM}{\hat M}

\newcommand{\tpr}{\widetilde {\text{pr}}}
\newcommand{\tB}{\widetilde B}
\newcommand{\tx}{\widetilde x}
\newcommand{\ty}{\widetilde y}
\newcommand{\txi}{\widetilde \xi}
\newcommand{\teta}{\widetilde \eta}
\newcommand{\tna}{\widetilde \nabla}
\newcommand{\tth}{\widetilde \theta}

\newcommand{\diml}{\text{dim}}
\newcommand{\var}{\varepsilon}
\newcommand{\End}{\text{End }}
\newcommand{\loc}{\text{loc}}
\newcommand{\Symp}{\text{Symp}}
\newcommand{\Sympo}{\text{Symp}(\omega)}
\newcommand{\lam}{\lambda}
\newcommand{\na}{\nabla}
\newcommand{\Hom}{\text{Hom}}
\newcommand{\Ham}{\text{Ham}}
\newcommand{\Ker}{\text{Ker}}
\newcommand{\dist}{\text{dist}}
\newcommand{\psl}{\rm{PSL}}
\newcommand{\rk}{\roman{rk }}

\newcommand{\psho}{\text{PSH}(\omega)}
\newcommand{\pshO}{\text{PSH}(\Omega)}

\renewcommand\qed{ }
\begin{titlepage}
\title{\bf The principle of least action in the space of K\"ahler potentials}
\author{L\'aszl\'o Lempert \thanks{Research partially  supported by NSF grant DMS 1764167\newline 2020 Mathematics Subject classification 32Q15, 32U15, 53C35, 58B20, 58E30, 70H99}\\ Department of  Mathematics\\
Purdue University\\West Lafayette, IN
47907-2067, USA}
\thispagestyle{empty}

\end{titlepage}
\date{}
\maketitle
\abstract
Given a compact K\"ahler manifold, the space $\cH$ of its (relative) K\"ahler potentials is an infinite dimensional Fr\'echet manifold, on which Mabuchi and Semmes have introduced a natural connection $\nabla$. We study certain Lagrangians on 
$T\cH$, in particular Finsler metrics, that are parallel with respect to the connection. We show that geodesics of 
$\nabla$ are paths of least action, and prove a certain convexity property of the least action. This generalizes earlier results 
of Calabi, Chen, and Darvas.
\endabstract

\section{Introduction}

Let $(X,\omega)$ be an $n$ dimensional, connected, compact K\"ahler manifold and
\[\cH=\cH_\omega=\{u\in C^\infty(X): \omega +dd^cu=\omega_u>0\}\]
its space of relative K\"ahler potentials. Here $C^\infty(X)$ refers to the Fr\'echet space of real valued smooth functions on $X$, and $d^c=i(\op-\partial)/2$, so that $dd^c=i\partial\op$. The space $\cH$, as an open subset of a Fr\'echet space, inherits a F\'echet manifold structure, whose tangent bundle has a canonical trivialization $T\cH\approx \cH\times C^\infty(X)$. Mabuchi and Semmes \cite{M,S} independently and with different motivations have introduced a torsion free connection $\na$ on $T\cH$. Mabuchi, as a tool to study special K\"ahler metrics, defined a Riemannian metric on $\cH$ and obtained $\na$ as the Levi--Civita  connection of the metric. Somewhat later Semmes found the connection in search 
for a geometric interpretation of interpolation of Banach spaces and of a certain homogeneous complex Monge--Amp\`ere equation associated with interpolation. He also determined all Riemannian metrics compatible with the connection: they are linear combinations of Mabuchi's metric and the square of a one form.

One way to explain $\na$ is through its parallel transport. We will use dot $\,\dot{ }\,$ to denote derivative of a function  of 
one real variable, and grad$_v$ to refer to gradient of a function $X\to\bR$ with respect to the K\"ahler metric of 
$\omega_v$. Let $u:[a,b]\to \cH$ be a smooth path. By integrating the time dependent vector field 
$(-1/2)$ grad$_{u(t)}\dot u(t)$ on $X$ we obtain a smooth family of diffeomorphisms $\varphi(t): X \to X$.  In fact 
$\varphi(t):(X, \omega_{u(0)})\to (X,\omega_{u(t)})$ is symplectomorphic. The parallel translate of 
$\xi\in T_{u(t)} \cH\approx C^\infty(X)$ to $u(0)$ along the path $u$ is then
\begin{equation}  
\xi\circ\varphi(t)\in C^\infty(X)\approx T_{u(0)} \cH.
\end{equation}

Understanding the geodesics of this connection was already marked in \cite{M,S} as an interesting and potentially important problem, and Donaldson's subsequent work \cite{Do} gave further impetus to study them. By now the boundary value problem for geodesics is well understood. On the one hand Darvas, Hu, Vivas, and myself \cite{D1, DL, Hu, LV} proved that points in $\cH$ cannot always be connected by a geodesic, not even if they are close to each other. On the other hand work by Berman--Demailly, Berndtsson, B\l ocki, Chen, Chen--Feldman--Hu, Chu--Tosatti--Weinkove, Darvas, and He \cite{BD, Be1, Bl1, Bl2, C, CFH, CTW, D2, D3, He} gave that the geodesic equation extends to various enlargements of $\cH$, and in these enlargements any pair of points, or at least nearby points, can be connected by solutions of the extended geodesic equation, weak geodesics. It follows from Chen's work that in those enlargements to which Mabuchi's metric extends, weak geodesics minimize length. In \cite{D2} Darvas generalized Mabuchi's metric to certain Orlicz type Finsler metrics on $\cH$, determined the metric completions of $\cH$ under these metrics, and again found that weak geodesics in these completions minimize length. In a slight overstatement the length minimizing paths are independent of which of Darvas's metric we use to compute length. This was surprising at first sight.

But in fact in geometry one encounters other similar phenomena. In a normed vector space straight line segments minimize length no matter what norm is chosen. There is also the analogy between $\cH$ and the space $\cQ$ of positive definite quadratic forms on $\bR^k$. $\cQ$ has a natural torsion free connection that turns it in a symmetric space $\approx\text{GL}^+_k(\bR)/\text{SO}_k(\bR)$; and for all parallel Finsler metrics---i.e. those that are invariant under 
$\text{GL}^+_k(\bR)$---the shortest paths are the same: subarcs of left translates of certain one parameter subgroups in 
$\text{GL}^+_k(\bR)$, projected to $\text{GL}^+_k(\bR)/\text{SO}_k(\bR)$. 

Now $\cH$ with Mabuchi's connection is itself a symmetric space \cite{Do, M, S}, at least according some definitions of a symmetric space (while it is not according to some others, \cite{L3}). Although there is no group acting transitively on\footnote {This follows from \cite{L2}. Even though Theorem 1.2 there is formulated for isometries of Mabuchi's metric, the proof, verbatim, gives that if $\omega_u$ is analytic while $\omega_v$ is not, then no diffeomorphism of $\cH$ can preserve $\na$ and map $u$ to $v$.} $(\cH,\na)$, the holonomy groupoid $\Gamma$ of $(\cH,\na)$ acts on $T\cH$. Thus $\Gamma=\bigcup_{u, v\in \cH}\Gamma_{uv}$, where $\Gamma_{uv}$ consists of linear isomorphisms $T_u\cH\to T_v \cH$ that arise as parallel transport along piecewise smooth paths from $u$ to $v$. Concatenation of parallel transports defines an operation $\Gamma_{uv}\times\Gamma_{vw}\to \Gamma_{uw}$ that turns $\Gamma$ in a grupoid. That a Finsler metric or a function $L:T\cH \to \bR$ is parallel means it is invariant under $\Gamma$.

Darvas's metrics are parallel. They are defined in terms of integrals
\begin{equation}  
\int_X\chi(\xi)\omega_u^n,\quad u\in \cH, \quad \xi\in T_u \cH\approx C^\infty(X),
\end{equation}
with a fixed ``Young weight'' $\chi:\bR\to [0,\infty]$, and are invariant under parallel transport simply because in the formula (1.1) for parallel transport, $\varphi(t)$ satisfies $\varphi(t)^*\omega_{u(t)}=\omega_{u(0)}$. But there are many parallel Finsler metrics on $T\cH$ beyond those considered in \cite{D2}. The simplest is, for given $0<\alpha<1$,
\[
p(\xi)=\sup\Big\{\int_E |\xi| \omega^n_u\, \Big / \, \Big(\int_E \omega^n_u\Big)^\alpha: E\subset X \text{ is measurable}\Big\},
\]
$\xi\in T_u \cH$. This is known as weak $L^q$-norm or $L^{q,\infty}$ Lorentz norm, $q=1/\alpha$.

Our thesis is that the proper generality of Darvas's results on his metrics is parallel, or holonomy invariant, Finsler metrics and more generally, fiberwise convex functions $T\cH\to\bR$, ``Lagrangians''. In this paper and in a sequel we will show 
that many of his results generalize to this framework. Most of the time we will consider Lagrangians on $T\cH$ that extend to the space $B(X)\cap\psho$ of bounded $\omega$--plurisubharmonic functions. Here we denote by $B(X)$ the 
Banach space of bounded Borel 
functions $\xi:X\to\bR$ with the norm $||\xi||=\sup|\xi|$; the Lagrangians of interest extend to 
$(B(X)\cap\psho)\times B(X)$. (The more common space $L^\infty(X)$ is a quotient of $B(X)$, but we have little 
use for it in this paper.)
A generalization of holonomy invariance can be defined for such functions. Our results pertain to invariant Lagrangians that are convex in the $B(X)$ variable and have a certain continuity property, that we call strong continuity (Definition 3.1).

\begin{thm}  
[=Theorem 8.1, Principle of least action]
If $v:[0,T]\to B(X)\cap \psho$ is a weak geodesic, and $C^1$ as a map into the Banach space $B(X)$, then it minimizes action
\[\int^T_0 L(\dot u(t)) dt\]
among all piecewise $C^1$ paths $u:[0,T] \to B(X)\cap \psho$ with $u(0)=v(0)$, $u(T)=v(T)$.
\end{thm}   

Weak geodesics may fail to be $C^1$ (Example 5.4), but from Chen's work \cite{C} we do know that a weak geodesic with endpoints in $\cH$ is $C^1$. The theorem can be proved for weak geodesics rather less regular than
$C^1$, but even this relaxed regularity is not automatic.

The next result is about how least action varies as one moves along weak geodesics; it is a manifestation of seminegative curvature. Fix $T>0$. If $w,w'\in B(X)\cap\psho$, the least action $\cL_T(w,w')$ between them is the 
infimum of the actions $\int^T_0 L(\dot u(t)) dt$ over all piecewise $C^1$ paths $u:[0,T]\to B(X)\cap \psho$ connecting 
$w$ with $w'$.
It is not obvious, but by Lemma 9.4, $\cL_T(w,w')$ is finite.

\begin{thm}   
[=Theorem 9.1]  If $u,v:[a,b]\to B(X)\cap\psho$ are weak geodesics, 
then the function $\cL_T(u,v):[a,b]\to\bR$ is convex.
\end{thm}   

We also prove a converse of sorts to Theorem 1.1: under certain  conditions, {\sl only} weak geodesics minimize
action, see Theorem 11.1. When $L$ defines Mabuchi's metric, Darvas already proved this, even for paths
more general than what our theorem covers \cite[Theorem 1]{D3}.

The tools of this paper are Chen's work on $\var$--geodesics, rudiments of Guedj--Zeriahi's pluripotential theory, and our results on invariant convex functions on $C^\infty(X)$ \cite{C, GZ1, GZ2, L4}. In
the proof of Theorem 1.1, even if the details are different, overall we will 
be able to follow the strategy of Calabi, Chen, and Darvas \cite{C, CC, D2, D3}. Once basic properties of our Lagrangians are established, the greater generality occasionally results in
less computation in the proofs for the following reason. Say, for a holonomy invariant Finsler metric $p:T\cH\to [0,\infty)$, there is a family $\cF\subset T\cH\approx \cH\times C^\infty(X)$ such that 
\begin{equation}   
p(\xi)= \sup\Big\{\int_X f\xi\omega^n_u : f\in \cF\cap T_u \cH\Big\},\quad \xi\in T_u\cH
\end{equation}
(see Theorem 2.4), and the integrals in (1.3), linear in $\xi$, can be easier to manipulate than the nonlinear integrals in (1.2).

It appears that the greatest generality in which action can be defined by an integral is the space of bounded  
$\omega$--plurisubharmonic functions. Nonetheless, action can be defined for any path in $\psho$ as a limit of integrals. Whether this action is finite or $\pm\infty$ of course depends on the path and on 
the Lagrangian. We plan to address this and 
related questions in a sequel to this paper.

Lagrangians even beyond Finsler metrics are not new to the subject. Chen's $\var$--geodesics are trajectories of a Lagrangian $L:T\cH\to\bR$ (albeit not holonomy invariant), with kinetic energy term the square of Mabuchi's metric and
potential energy a multiple of $V(u)=-\int_X u\omega^n$. Functions on $\cH$ that its geometry motivates, and that are used in existence problems in K\"ahler geometry, are also not new. Aubin's functional $I:\cH\to\bR$ \cite[p.146]{Au}, 
\[I(u)=\int_X u(\omega^n-\omega_u^n)\]
is a constant multiple of the total geodesic curvature of the line segment $[0,1]\ni t\mapsto tu\in \cH$, measured in Darvas's 
$L^1$ Finsler metric. Monge--Amp\`ere energy also arises from the geometry of $\cH$. It is a convex
function on $\cH$, for example in the sense that its restrictions to geodesics of $\na$ are convex; 
but its negative is also convex and, up to scaling and adding a constant, it is the only continuous function
that has this property.

We hope that a geometrical approach to functions on $\cH$ and on related spaces, in the spirit of this paper, will be of use in analytical problems on K\"ahler manifolds.

Contents. Section 2 is about basic properties of holonomy invariant convex Lagrangians $T\cH\to\bR$. Section 3 is about a subclass of Lagrangians that have an extra continuity property, which makes it possible to extend them to a larger vector bundle. Many of the results in these sections are direct consequences of results in \cite{L4}. Section 4 reviews the notion of weak and $\var$--geodesics, and $\var$--Jacobi fields. Section 5 introduces the action and formulates Theorems 1.1 and 1.2 in precise forms. It also gives a road map to their proofs, which occupy sections 6--9. Section 10 provides a
discretized formula for action, which suggests how to generalize to less regular paths, and section 11 addresses the
problem of uniqueness of paths that minimize action. 

In this paper we freely use basic notions of infinite dimensional analysis and geometry. There are many sources the reader can consult on these matters, one of them \cite{L1}, written with an eye on the space $\cH$ of K\"ahler potentials.

Acknowledgement. During the preparation of this paper I have profited from pluripotential theoretic discussions with 
Darvas and Guedj.

\section{Lagrangians}  

The central objects of this paper are continuous functions $L:T\cH\to\bR$ that are convex on each tangent space $T_u \cH$ and have a certain invariance property; as well as the associated action functional 
$\cL(u)=\int_a^b L(\dot{u}(t))dt\,\big(=\int^b_a L\circ\dot u$ for brevity). In this section and in the next we record basic facts about such functions which follow more or less directly from \cite{L4}, that dealt with the action on $C^\infty(X)$
of Hamiltonian diffeomorphisms of $(X,\omega)$ and with invariant convex functions on $C^\infty(X)$. 
As explained in the Introduction, for $L$ 
the invariance property in question is invariance under the holonomy grupoid $\Gamma$ of $(\cH,\na)$. Thus, if $\xi_1\in T_{u(1)} \cH$ is the parallel translate of $\xi_0\in T_{u(0)}\cH$ along a piecewise smooth path $u:[0,1]\to \cH$, then $L(\xi_0)= L(\xi_1)$. This property in fact implies a much stronger and more primitive notion of invariance.

\begin{defn}   
Given two measure spaces $(X,\mu)$ and $(Y,\nu)$, we say that measurable functions $\xi:X\to\bR$ and $\eta:Y\to\bR$ are equidistributed, or are strict rearrangements of each other, if $\mu(\xi^{-1}B)=\nu(\eta^{-1}B)$ for every Borel set $B\subset \bR$.
\end{defn}   

In finite measure spaces this is equivalent to requiring $\mu(\xi>t)=\nu(\eta>t)$ for all $t\in\bR$.

Back to our $n$ dimensional K\"ahler manifold $(X,\omega)$, if $u\in \cH$ we let $\mu_u$ denote the measure induced by $\omega^n_u$. Given measurable $\xi,\eta: X\to\bR$ we will write 
\begin{equation}    
(\xi,u)\sim(\eta,v) \text{ if } \xi,\eta \text{ are equidistributed as functions on }(X,\mu_u), (X,\mu_v).
\end{equation}
 When smooth $\xi,\eta$ are viewed as tangent vectors in $T_u \cH$, $T_v \cH$, we will just write $\xi\sim \eta$.

\begin{thm}     
A function $L:T\cH\to\bR$, continuous and convex on each fiber $T_u \cH$, is invariant under the holonomy gruppoid $\Gamma$ if and only if it is invariant under strict rearrangements: $L(\xi)=L(\eta)$ when $\xi\sim\eta$.
\end{thm}     

For the proof we need to understand the holonomy groups $\Gamma_{uu}$. (1.1) shows that in general, elements of 
$\Gamma_{uv}$, isomorphisms $T_u \cH\to T_v \cH$, are pullbacks by certain symplectomorphisms 
$\varphi:(X,\omega_v)\to (X, \omega_u)$. Let us write $G$ for those symplectomorphisms that induce elements of 
$\Gamma_{00}$. Thus $G$ is a subgroup of the Fr\'echet--Lie group Diff$\,X$ of diffeomorphisms of $X$.

\begin{lem}    
The closure of $G$ in Diff$\,X$ contains all Hamiltonian diffeomorphisms of $(X,\omega)$.
\end{lem}    

Recall that Hamiltonian diffeomorphisms are time--1 maps of time dependent Hamiltonian vector fields sgrad$\,\zeta_t$, i.e., vector fields that are symplectic gradients with $\zeta_t\in C^\infty(X,\omega)$ a smooth family, $t\in[0,1]$.

\begin{proof}
Let $\frak g$ (the ``Lie algebra'' of $G$) consist of smooth vector fields $V$ on $X$ for which there is a smooth map $\varphi:[0,1]\to G\subset \text{Diff}\, X$ such that $\varphi(0)=\text{id}_X$ and $\dot\varphi(0)=d\varphi(t)/dt|_{t=0}=V$. This is a vector subspace of the space of all vector fields: for example, if $\varphi,\psi$ realize vector fields $V, W\in\frak{g}$, then $\varphi(t)\circ\psi(t)$ realizes $V+W$. In \cite[pp. 512-513]{S} Semmes essentially proved that $\frak g$ contains all Hamiltonian vector fields sgrad$\,\zeta$. Essentially only, because the proof of his Lemma 4.1 is given only in Sobolev spaces, not in $C^\infty(X)$. At any rate, we will need a slightly stronger, parametrized statement, to wit: If $\zeta:[a,b]\to C^\infty(X)$ is smooth, then there is a smooth family
\begin{equation}   
[a,b]\times[0,1]\ni (s,t) \mapsto \varphi^s_t\in G\subset \text{Diff}\,X
\end{equation}
such that $\varphi^s_0=\text{id}_X$ and $\partial_t\varphi^s_t |_{t=0}=\text{sgrad} \,\zeta(s)$ for all $s$.

To verify this, recall Semmes' construction in \cite[top of p. 512]{S} that, given $\xi,\eta\in C^\infty(X)$, shows that the 
Poisson bracket 
$\{\xi,\eta\}\in C^\infty(X)$, determined by $\omega$, has symplectic gradient in $\frak g$. The same construction works 
with a parameter appended. Thus, if $\xi,\eta:[a,b]\to C^\infty(X)$ are smooth, there is a smooth family $\varphi^s_t\in G$ as in (2.2), $\varphi_0^s=\text{id}_X$ and $\partial_t\varphi^s_t |_{t=0}=\text{sgrad} \{\xi(s),\eta(s)\}$. But any smooth $\zeta:[a,b]\to C^\infty(X)$ such that $\int_X\zeta(s)\omega^n=0$ can be written
\begin{equation}    
\zeta(s)=\sum^m_{j=1} \{\xi_j(s),\eta_j(s)\}, \quad m=4n+1,
\end{equation}
with suitable smooth $\xi_j, \eta_j:[a,b]\to C^\infty(X)$. In fact $\xi_j$ can be chosen constant, and arbitrary as long as 
$\xi_j(s)\equiv\oxi_j$ embed $X$ into $\bR^m$. The statement, without $s$--dependence, corresponds to 
\cite[Lemma 4.1]{S}, but was already proved in \cite{AG}. Atkin and Grabowski's proof is easily modified to provide (2.3). The proof of \cite[(5.2) Theorem]{AG} depends on \cite[(2.6) Proposition]{AG}, the $s$--dependent version of which says that if $\xi_j\in C^\infty(X)$,  $j=1,\dots,m$, embed $X$ into $\bR^m$, then any smooth family 
$\psi^s$ of smooth $k$-forms on $X$, $s\in [a,b]$, can be written 
\[\psi^s=\sum_{i_1, i_2,\dots} f_{i_1\dots i_k}(s) d\xi_{i_1}\wedge\dots\wedge d\xi_{i_k}\]
with $f_{i_1\dots i_k} :[a,b] \to C^\infty(X)$ smooth. This is proved by an obvious cohomology vanishing as in \cite{AG}. Another ingredient of the proof of \cite[(5.2) Theorem]{AG}, on p. 325 there, in $s$--dependent version says that given a smooth family $\alpha^s$ of exact smooth forms on $X$, there is a smooth family $\beta^s$ of smooth forms such that $d\beta^s=\alpha^s$.
One way to prove this is by Hodge theory, which gives that the unique solution $\beta^s$ of $d\beta^s=\alpha^s$ that is othogonal to Ker$\,d$ depends smoothly on $s$. The rest of the proof in \cite{AG} manipulates identities, and changes not if a parameter $s$ is appended. Thus (2.3) is proved.

We can now construct $\varphi^s_t\in G$ as in (2.2). First, subtracting from $\zeta$ a smooth function $c:[a,b]\to\bR$ we obtain $\zeta':[a,b]\to C^\infty(X)$ with $\int_X\zeta'(s)\omega^m=0$. We find $\xi_j, \eta_j$ as in (2.3), corresponding to $\zeta'$ rather than $\zeta$, and then smooth maps $(s,t)\mapsto \varphi^s_{jt}\in G$ such that $\varphi^s_{j0}=\text{id}_X$ and $\partial_t\varphi^s_{jt}=\{\xi_j(s),\eta_j(s)\}$ at $t=0$. Since $\text{sgrad}\,\zeta=\text{sgrad}\,\zeta'$, the diffeomorphisms
\[\varphi^s_t=\varphi^s_{1t}\circ \varphi^s_{2t}\circ\dots\circ\varphi^s_{mt}\]
have $t$--derivative $\text{sgrad}\,\zeta(s)$ at $t=0$.

After these preparations we are ready to prove the lemma. 
Suppose $\varphi^1:(X,\omega)\to (X,\omega)$ is a Hamiltonian diffeomorphism. This means it can be included in the flow $\varphi^s$ of Hamiltonian vector fields $V^s=\text{sgrad}\,\zeta(s)$, 
\begin{equation}     
\partial_s\varphi^s=V^s(\varphi^s),\quad 0\le s\le 1, \quad \varphi^0=\text{id}_X.
\end{equation}
Here $\zeta:[0,1]\to C^\infty(X)$ is smooth. The $\varphi^s_t$ constructed above for this $\zeta$ can be used as integrators in a 1--step scheme to approximate the solution of the initial value problem (2.4). General theory gives that 
\begin{equation}       
\varphi^{(k-1)/k}_{1/k}\circ \varphi^{(k-2)/k}_{1/k}\circ \dots\circ\varphi^0_{1/k}\to \varphi^1\quad \text{ in } C^\infty(X)
\end{equation}
as $k\to \infty$.

(Details are as follows. Smoothly embed $X$ in some $\bR^m$ and with $p\in\bN$, view $\varphi^s$ as an element of the Banach space $B=C^p(X)\times\dots\times C^p(X)$, $m$ copies of $C^p(X)$. Extend $\varphi^s_t: X\to X$ to a smooth family of maps $\psi^s_t:\bR^m\to\bR^m$ and extend $V^s$ to a vector field on $\bR^m$ by $V^s=\partial_t\psi^s_t |_{t=0}$.
The error analysis of e.g. \cite[p.160, Theorem 3.4]{HNW}, or more directly \cite[Theorem 4.1, Corollary 4.2]{An}, gives that 
\begin{equation}       
\psi^{(k-1)/k}_{1/k}\circ \psi^{(k-2)/k}_{1/k}\circ\dots\circ \psi^0_{1/k}\circ\varphi^0\to \varphi^1\quad \text{ in } B
\end{equation}
as $k\to\infty$. Both \cite{HNW, An} work in finite dimensional Banach spaces, the latter in $\bC^m$, but the same reasoning proves the result in any Banach space.)

Since the left hand side of (2.6) is $\varphi^{(k-1)/k}_{1/k}\circ\dots\varphi^0_{1/k}\in G$, we proved that $\varphi^1$ 
is indeed in the closure of $G$.
\end{proof}

\begin{proof}[Proof of Theorem 2.2.]
 That invariance under strict rearrangements implies holonomy invariance follows since parallel transport is realized by composition with a symplectomorphism, and such compositions send functions to their strict rearrangements. The converse implication depends on Lemma 2.3. This implies that $L(\xi)=L(\xi\circ\varphi)$ if $\xi\in T_0 \cH$ and $\varphi\in\text{Diff}(X,\omega)$ is Hamiltonian. By \cite[Theorem 1.2]{L4}, $L |T_0\cH$ is therefore invariant under strict rearrangements. To complete the proof, take $\xi\in T_u \cH$, $\eta\in T_v \cH$ such that $\xi\sim\eta$. Parallel translate $\xi,\eta$ to $\xi', \eta'\in T_0 \cH$ along arbitrary smooth paths. Then $\xi'\sim\xi\sim\eta\sim\eta'$, whence $L(\xi)=L(\xi')=L(\eta')=L(\eta)$.
\end{proof}

In what follows, a fiberwise continuous and convex function $L:T\cH\to \bR$ that is invariant under strict rearrangements will be called an invariant convex Lagrangian. The chief device to analyze their finer properties is the following representation theorem. We write $B(X)$ or $B(X,\mu)$---when a Borel measure $\mu$ on $X$  plays a role---for the Banach space of bounded Borel functions on $X$, with the supremum norm $\|\ \ \|$. 

\begin{thm}     
Given an invariant convex Lagrangian $L: T\cH\to\bR$, there are families $\cA_u\subset\bR\times B(X)$, $u\in \cH$ such that for $\xi\in T_u \cH\approx C^\infty(X)$
\begin{equation}    
L(\xi)=\sup_{(a,f)\in\cA_u} a+\int_X f\xi\omega^n_u.
\end{equation}
$\cA_u$ can be chosen in $\bR\times C^\infty(X)$, and have the property that whenever $(a,f)\in \cA_u$ and $\varphi:(X,\omega_v)\to (X, \omega_u)$ is a symplectomorphism, then $(a,f\circ\varphi)\in \cA_v$. Alternatively, $\cA_u$ can be chosen to be strict rearrangement invariant: if $f\in B(X,\mu_u)$ and $g\in B(X,\mu_v)$ are equidistributed, and
$(a,f)\in\cA_u$, then $(a,g)\in \cA_v$.

If $L$ is also positively homogeneous, $(L(c\xi)=L(\xi)$ whenever $c\in (0,\infty))$, then in addition $\cA_u$ can be chosen in $\{0\}\times C^\infty(X)$, respectively, in $\{0\}\times B(X).$
\end{thm}

\begin{proof}
Most of the proof was done in \cite{L4}. Lemma 2.1 there produces $\cA_0\subset\bR\times C^\infty(X)$ that satisfies (2.7) when $u=0$. If we adjoin to $\cA_0$ all pairs $(a,f\circ\varphi)$ with $(a,f)\in\cA_0$ and $\varphi:(X,\omega)\to (X,\omega)$ a symplectomorphism, because of the invariance of $L$ the supremum in (2.7) is not going to change (for $u=0)$. So we can assume that $\cA_0$ already is invariant under symplectomorphisms. We then define $\cA_u$ to consist of pairs $(a,f\circ\psi)$ with $(a,f)\in\cA_0$ and $\psi:(X, \omega_v)\to (X,\omega)$ a symplectomorphism. This will do, since if $\xi\in T_u \cH$, with the above $\psi$
\begin{align*}
\sup_{(a,g)\in \cA_u}& a+\int_X g\xi \omega^n_v=\sup_{(a,g)\in\cA_u}a+ \int_X(g\circ\psi^{-1})(\xi\circ\psi^{-1})\omega^n\\
=&\sup_{(a,f)\in\cA_0} a+\int_X (\xi\circ\psi^{-1}) f\omega^n=L(\xi\circ \psi^{-1})=L(\xi).
\end{align*}

Alternatively, we can modify the above $\cA_u$ to $\cA'_u$ consisting of all $(a, g)\in\bR\times B(X)$ for which there is 
$(a,f)\in\cA_0$ such that $(f,\mu_0)\sim(g,\mu_u)$.This will not change the supremum 
in (2.7), with $\cA'_u$ now, either. It suffices to check this for $u=0$. By a variant of a lemma of Katok, 
\cite[Lemma 3.2]{L4}, if $(f,\mu_0)\sim (g,\mu_0)$ then there is a sequence $\varphi_k:(X,\omega)\to (X,\omega)$ of 
symplectomorphisms such that $\int_X|g-f\circ\varphi_k|\omega^n\to 0$. Therefore
\begin{align*}
\int_X g\xi\omega^n&=\lim_{k\to\infty} \int_X(f\circ \varphi_k)\xi\omega^n=\lim_{k\to\infty}\int_X(\xi\circ\varphi^{-1}_k)f\omega^n,\quad \text{and so }\\
&a+\int_X g\xi\omega^n\le \lim_{k\to\infty} L(\xi\circ\varphi^{-1}_k)=L(\xi).
\end{align*}
Thus replacing $\cA_u$ with $\cA'_u$, (2.7) will still hold, and $\cA'_u$ is now strict rearrangement invariant.

Finally, if $L$ is positively homogeneous, the statement of the theorem follows in the same way from the corresponding part of \cite[Lemma 2.1]{L4}.
\end{proof}

The Lagrangians in this section were required to be continuous on the fibers of $T\cH$. But, coupled with invariance, this implies continuity on $T\cH$:

\begin{thm}   
An invariant convex Lagrangian $L:T\cH\to\bR$ is a continuous function on the Fr\'echet manifold $T\cH$.
\end{thm}

\begin{proof}
Suppose $u,u_k\in \cH$, $\xi \in T_u\cH$, $\xi_k\in T_{u_k} \cH$, and $\xi_k\to\xi$. This simply means that as elements of $C^\infty(X)$, $u_k\to u$ and $\xi_k\to\xi$. Parallel translate $\xi_k$ to $\eta_k\in T_u \cH$ along the straight line segment $t\mapsto tu+(1-t)u_k$. This is done by integrating the time dependent vector field (1/2) grad$_{tu+(1-t)u_k}(u_k-u)$ on $X$, for $0\le t\le 1$. If the time--1 map is $\psi_k:X\to X$,
then $\eta_k=\xi_k\circ\psi_k$. Since $\psi_k\to \text{id}_X$ in the $C^\infty$ topology, $\eta_k\to\xi$ in 
$C^\infty(X)\approx\ T_u\cH$, as $k\to\infty$. Hence $\lim_k L(\xi_k)=\lim_kL(\eta_k)=L(\xi)$, as claimed.
\end{proof}

\section{Extending Lagrangians}  

As said in the Introduction, weak geodesics tend not to stay in the space $\cH$. Therefore, even in order to formulate a principle of least action we need to evaluate the action of a Lagragian along paths in spaces larger than $\cH$. In this section we will extend invariant convex Lagrangians $T\cH\to\bR$ to larger Banach bundles and describe properties of the extended Lagrangians.

We start by recalling definitions. Let $Y$ be a complex manifold and $\Omega$ a smooth real $(1,1)$ form on it, $d\Omega=0$. A function $u:Y\to [-\infty,\infty)$ is $\Omega$--plurisubharmonic if $\rho+u$ is plurisubharmonic whenever 
$\rho$ is a potential of $\Omega$ over some open $V\subset Y$, i.e., $\Omega|V=dd^c\rho$. We use the convention that 
$\equiv -\infty$ is not plurisubharmonic, and write $\pshO$ for the set of $\Omega$--plurisubharmonic functions. Back to our K\"ahler manifold $(X,\omega)$, we denote by $\cE(\omega)$ the class of $u\in\psho$ with full Monge--Amp\`ere mass, see \cite{GZ1}. The 
class contains all bounded $\omega$--plurisubharmonic functions. The Monge--Amp\`ere measure on $X$, corresponding to $\omega^n_u$, will again be denoted $\mu_u$. This is a Borel measure on $X$, its crucial property is $\mu_u(X)=\int_X\omega^n$. We endow $\cE(\omega)$ with the discrete topology, and let
\begin{equation}   
T^\infty\cE(\omega)=\cE(\omega)\times B(X),
\end{equation}
a trivial Banach bundle with fibers the bounded Borel functions on $X$. Corresponding to usage in the subject we will not distinguish between elements $\xi\in T_u^\infty\cE(\omega)$ and their representation $\xi\in B(X)$ in the trivialization (3.1). The embedding $C^\infty(X)\hookrightarrow B(X)$ induces an embedding $T\cH\hookrightarrow T^\infty\cE(\omega)$ of vector bundles, continuous if $\cH$ is considered with the discrete topology. We will also deal with a bundle between
the image of $T\cH$ and $T^\infty\cE(\omega)$, 
\[
T^c\cH=\cH\times C(X).
\]

\begin{defn}   
Suppose $u\in\cE(\omega)$ and $V\subset B(X,\mu_u)$ is a vector subspace. We say that a function 
$p:V\to\bR$ is strongly continuous if  $p(\xi_k)$  converges whenever  $\xi_k\in V$ is a uniformly 
bounded sequence that converges $\mu_u$ almost everywhere.
\end{defn}   
In this case $\lim_kp(\xi_k)$ depends only on $\xi=\lim_k\xi_k$, since another sequence $\eta_k\to\xi$ can be combined with $\xi_k$ into one sequence.

\begin{thm}    
Any invariant convex Lagrangian $L:T\cH\to\bR$  has a unique fiberwise continuous extension
to $T^c\cH$. This extension is strict rearrangement invariant and fiberwise convex. If in addition $L$
is strongly continuous on the fibers $T_u\cH$, then it has a unique extension 
to $T^\infty\cE(\omega)$ that is strict rearrangement invariant, and strongly continuous on the fibers 
$T^\infty_u\cE(\omega)$. This extension is 
 fiberwise convex. 
\end{thm}
For example, Darvas's  metrics in \cite{D2}, coming from finite Young weights, cf. (1.2), are strongly continuous on the fibers.---The proof  of the theorem will use the notion of decreasing  rearrangement of measurable functions 
$\eta:(Y,\nu)\to \bR$ on a measure space. This is a decreasing, upper semicontinuous function $\eta^\star:[0,\nu(Y)]\to\bR$, equidistributed with $\eta$. Thus $\nu(s\le\eta\le t)$ is equal to the length of the maximal interval on which 
$s\le\eta^\star\le t$. The requirement of upper semicontinuity for the decreasing function $\eta^\star$ translates to left continuity, which 
differs from the more usual right continuity requirement, but the difference is of no consequence. In our setting 
\begin{equation}   
\nu(\eta\ge \eta^\star(s))=s,
\end{equation}
and more generally,
\begin{equation}
\nu(\eta\ge t)\le\tau  \,\, \text{implies}\,\, \eta^\star(\tau)\le t,\qquad
\nu(\eta\ge t)\ge\tau  \,\, \text{implies}\,\, \eta^\star(\tau)\ge t.
\end{equation}   
When $\xi\in T_u^\infty\cE(\omega)$, we compute $\xi^\star$ with respect to the measure $\mu_u$.

\begin{proof}[Proof of Theorem 3.2]

Suppose first $L:T \cH\to\bR$ is just continuous. By \cite[Theorem 5.2]{L4} each $L|T_u\cH$ has a unique continuous 
extension $C(X)\to\bR$. These extensions are convex and strict rearrangement 
invariant, and together define the extension $\oL:T^c\cH\to\bR$ of $L$. To see that $\oL$ is strict rearrangement invariant,
given equidistributed $\xi\in T_u^c\cH$ and $\eta\in T^c_v\cH$, choose $\xi_j\in C^\infty(X)$ converging uniformly to $\xi$.
Let $\varphi\colon(X,\omega_v)\to(X,\omega_u)$ be a symplectomorphism (for example one that induces parallel transport
along some path connecting $u$ and $v$). Then $\xi\circ\varphi,\eta\in T_v^c\cH$ are equidistributed, whence
\[
\oL(\eta)=\oL(\xi\circ\varphi)=\lim_{j\to\infty} L(\xi_j\circ\varphi)=\lim_{j\to\infty} L(\xi_j)=\oL(\xi).
\]

As to the second case of the theorem, if $L$ is strongly continuous, by \cite[Theorem 5.2]{L4} 
$L|T_0 \cH: C^\infty(X)\to\bR$ has a unique strongly continuous and strict rearrangement 
invariant extension $q:B(X)\to\bR$; 
this extension is convex. Further to extend $q$ to $\oL:T^\infty\cE(\omega)\to\bR$, take a
$\xi\in T^\infty_u\cE(\omega)\simeq B(X,\mu_u)$. There is a measure preserving
$\theta:(X,\mu_0)\to[0,\mu_0(X)]$, the latter endowed with Lebesgue measure, see e.g.  \cite[Lemma 5.5]{L4}. Now let 
$\eta=\xi^\star\circ\theta\in T^\infty_0\cE(\omega)$ and
$\oL(\xi)=q(\eta)$. This is clearly the only strict rearrangement invariant way to extend $q: T^\infty_0\cE(\omega)\to\bR$ to $\oL: T^\infty\cE(\omega)\to\bR$, and  it is immediate that $\oL$ thus constructed 
has the properties claimed in the theorem.
\end{proof}

Further down we will not distinguish between an invariant convex Lagrangian $T\cH\to\bR$ that is (strongly) continuous 
on the fibers and its extension to $T^c\cH$, respectively, $T^\infty\cE(\omega)$ provided by Theorem 3.2; we 
will just refer to a (strongly continuous,) invariant, convex Lagrangian $L:T^c\cH\to\bR$ or $L:T^\infty\cE(\omega)\to\bR$.

\begin{lem}   
A strongly continuous, invariant, convex Lagrangian $L:T^\infty\cE(\omega)\to\bR$ is equi--Lipschitz continuous on bounded subsets of the fibers $T_u\cE(\omega)$ in the sense that given $R\in(0,\infty)$ there is an $A\in(0,\infty)$ such that for $u\in\cE(\omega)$ and $\xi,\eta\in T^\infty_u\cE(\omega)$
\begin{equation}    
\text{if}\quad \|\xi\|, \|\eta\| <R\quad \text{then}\quad |L(\xi)-L(\eta)|\le A \|\xi-\eta\|.
\end{equation}
\end{lem}

\begin{proof}
According to \cite[Theorem 5.4]{L4} (3.4) holds when $u=0$. The same $A$ will work for any $u$, for with a measure preserving $\theta:(X,\mu_0)\to [0,\mu_0(X)]$ as in the proof of Theorem 3.2 and $\xi,\eta\in T^\infty_u\cE(\omega)$
\[|L(\xi)-L(\eta)|=|L(\xi^\star\circ\theta)-L(\eta^\star\circ\theta)|\le 
A\sup |\xi^\star\circ\theta-\eta^\star\circ \theta| \le A \|\xi-\eta\|.\]
\end{proof}

Although we have endowed $\cE(\omega)$ with the discrete topology, we will need a continuity property of Lagrangians 
$T^\infty\cE(\omega)\to\bR$ stronger than fiberwise. This will involve the notion of Monge--Amp\`ere capacity cap of subsets 
of $X$ \cite{BT,K,GZ2}. Recall that a function $\xi:X\to\bR$ is quasicontinuous if for every $\var>0$ there is an open
$G\subset X$ of capacity cap$\,(G)<\var$ such that $\xi|X\setminus G$ is continuous; and a sequence of functions 
$\xi_j:X\to\bR$ converges to $\xi:X\to\bR$ in capacity if $\lim_{j\to\infty}\text{cap}(|\xi_j-\xi|>\delta)=0$ for every $\delta>0$.
In particular, a uniformly convergent sequence converges in capacity.

\begin{lem}   
Let $L:T^\infty\cE(\omega)\to\bR$ be strongly continuous, invariant, and convex.
Suppose $u_k\in\cE(\omega)$ either decrease, or uniformly converge, to a bounded $u\in\cE(\omega)$ as $k\to\infty$, and
uniformly bounded
$\xi_k\in T^\infty_{u_k}\cE(\omega)\approx B(X)$ converge in capacity to $\xi\in T^\infty_u\cE(\omega)\approx B(X)$. 
If $\xi$ is quasicontinuous, then $\xi_k^\star\to\xi^\star$ away from a countable subset of $[0,\mu_0(X)]$, and
$\lim_k L(\xi_k)=L(\xi)$.  
\end{lem}

\begin{proof}
If needed, we drop finitely many $u_k$ to arrange that the remaining $u_k$ are uniformly bounded.
Upon adding a constant to the $u_k$ and scaling $u,u_k,\omega$, and $L$, we can even arrange that $0\le u, u_k\le 1$.
Suppose first the $u_k$ decrease.  Define decreasing functions $f,g$ on $[0,\mu_0(X)]$
\[
f=\liminf_{k\to\infty}\xi^\star_k\le \limsup_{k\to\infty}\xi^\star_k=g.
\]
Let $s\in (0,\mu_0(X))$, $S=\xi^\star(s)$, and $\varepsilon>0$. With $A_k=\{|\xi_k-\xi|\ge\var\}$, $k\in\bN$,
\[
\{\xi_k\ge S+\var\}\subset \{\xi\ge S\}\cup A_k\quad\text{ and }\quad
\{\xi>S-\var\}\subset\{\xi_k\ge S-2\var\}\cup A_k.
\]
For $j\in\bN$ define continuous functions $F_j,G_j:\bR\to[0,1]$
\[
F_j(t)=\begin{cases}0 \text{ if } t\le S-1/j\\1\text{ if } t\ge S\\ \, \text{linear in between},\end{cases}\qquad 
G_j(t)=\begin{cases}0 \text{ if } t\le S-\var\\1\text{ if } t\ge S-\var+1/j\\ \,\text{linear in between.}\end{cases}
\]
Note that $F_j$ increases, $G_j$ decreases with increasing $j$. We can estimate
\begin{equation}\begin{gathered}
\mu_{u_k}(\xi_k\ge S+\varepsilon)\le\mu_{u_k}(\xi\ge S)+\mu_{u_k}(A_k)\le \int_X F_j\circ\xi\,d\mu_{u_k}+\text{cap}\,(A_k)\\
\mu_{u_k}(\xi_k\ge S-2\var)\ge\mu_{u_k}(\xi>S-\varepsilon)-\mu_{u_k}(A_k)\ge 
\int_X G_j\circ\xi\,d\mu_{u_k}-\text{cap}\,(A_k).
\end{gathered}\end{equation}
Since $F_j\circ\xi$, $G_j\circ\xi$ are quasicontinuous, by \cite[Theorem 4.26, Proposition 4.25]{GZ2}
\[
\lim_{k\to\infty}\int_XF_j\circ\xi\,d\mu_{u_k}=\int_XF_j\circ\xi\,d\mu_u,\qquad
\lim_{k\to\infty}\int_XG_j\circ\xi\,d\mu_{u_k}=\int_XG_j\circ\xi\,d\mu_u.
\]
Therefore, letting first $k\to\infty$ in (3.5), then $j\to\infty$, and using the monotone convergence theorem as well,
\begin{align*}
\limsup_{k\to\infty}\mu_{u_k}(\xi_k\ge S+\var)&\le\mu_{u}(\xi\ge S)\\
\liminf_{k\to\infty}\mu_{u_k}(\xi_k\ge S-2\var)&\ge\mu_u(\xi>S-\var)\ge\mu_u(\xi\ge S).
\end{align*}

Now $\mu_u(\xi\ge S)=s$, see (3.2). Hence, given $\sigma<s<\rho$, for sufficiently large $k$
\[
\mu_{u_k}(\xi_k\ge S+\varepsilon)<\rho,\qquad
\sigma< \mu_{u_k}(\xi_k\ge S-2\varepsilon).
\] 
Let $\eta=\xi_k$ and apply (3.3) first with $\nu=\mu_{u_k}$,  $t=S+\varepsilon$, and 
$\tau=\rho$; then with $\nu=\mu_u$, $t=S-2\varepsilon$, and $\tau=\sigma$.
We conclude $\xi^\star_k(\rho)\le S+\var$ and 
$\xi^\star_k(\sigma)\ge S-2\var$. In the limit $k\to\infty$
\[
g(\rho)\le S+\var=\xi^\star(s)+\varepsilon\quad\text{and}\quad f(\sigma)\ge S-2\var=\xi^\star(s)-2\varepsilon,
\]
for all $\sigma<s<\rho$.
If $f,g$ are continuous at $s$---which occurs apart from countably many values---, $g(s)\le \xi^\star(s)\le f(s)\le g(s)$ follows, i.e., $\lim_k \xi^\star_k(s)=\xi^\star(s)$ as claimed.

It is now easy to finish the proof. With a measure preserving $\theta:(X,\mu_0)\to [0, \mu_0(X)]$, as in the proof of Theorem 3.2, $\xi_k\in B(X, \mu_{u_k})$ and $\xi^\star_k\circ \theta\in B(X,\mu_0)$ are equidistributed, and $\xi^\star_k\circ\theta\to \xi^\star\circ\theta$ $ \mu_0$--almost everywhere. Hence 
\[
\lim_k L(\xi_k)=\lim_k L(\xi^\star_k\circ\theta)=L(\xi^\star\circ\theta)=L(\xi).
\]
We are done if $u_k$ are known to decrease.

Now suppose that $u_k$ converge uniformly. It suffices to prove that a subsequence of $L(\xi_k)$ converges to $L(\xi)$, and for this reason we can assume that $\|u_k-u_{k-1}\| < 2^{-k}$ for $k=2,3,\dots$. Then the sequence $v_k=u_k+2^{-k}$ decreases to $u$, and $\mu_{u_k}=\mu_{v_k}$. We can view $\xi_k\in T^\infty_{u_k}\cE(\omega)\approx B(X, \mu_{u_k})$ as elements $\xi'_k\in T^\infty_{v_k} \cE(\omega)\approx B(X,\mu_{v_k})$, which are strict rearrangements of $\xi_k$. Hence $L(\xi_k)=L(\xi'_k)\to L(\xi)$ by the first part of the proof.
\end{proof}

Later on we will need a variant of Lemma 3.4 in which the condition on $\xi_k$ is relaxed.

\begin{defn}  
We will say that a family $N$  of finite Borel measures on $X$ is hereditarily tight if for every open $U\subset X$ the restrictions
$\nu|U$, $\nu\in N$ are tight, i.e., given $\var>0$, there is a compact $K\subset U$ such that $\nu(U\setminus K)<\var$
for all $\nu\in N$.
\end{defn}

For example, if all $\nu\in N$ 
are absolutely continuous with respect to a finite Borel measure $\mu$, and the Radon--Nikodym 
derivatives  $d\nu/d\mu$ are uniformly bounded, then $N$ is hereditarily tight. 

\begin{lem}   
Let $L:T^\infty\cE(\omega)\to\bR$ be strongly continuous, invariant, and convex, and
$u_k\in\cE(\omega)$ either decrease, or uniformly converge, to a bounded $u\in\cE(\omega)$ as $k\to\infty$. Suppose
that the family $\mu_{u_k}$, $k\in\bN$, is hereditarily tight.
If uniformly bounded $\xi_k\in T^\infty_{u_k}\cE(\omega)\approx B(X)$ converge $\mu_u$--almost everywhere to 
$\xi\in T^\infty_u\cE(\omega)\approx B(X)$, then $\xi_k^\star\to\xi^\star$ away from a countable subset of 
$[0,\mu_0(X)]$, and $\lim_k L(\xi_k)=L(\xi)$.
\end{lem}
\begin{proof} 
As in the previous proof, it suffices to argue when the $u_k$ decrease. Let $\var>0$. First we claim that there are
uniformly bounded $\eta_k\in T^\infty_{u_k}\cE(\omega)$ and continuous $\eta\in T^\infty_{u}\cE(\omega)$   such that 
$\eta_k\to\eta$ uniformly and for $\var\le s\le\mu_0(X)-\var$
\begin{equation} 
\eta^\star(s+\var)\le\xi^\star(s)\le\eta^\star(s-\var), \qquad  \eta_k^\star(s+\var)\le\xi_k^\star(s)\le\eta_k^\star(s-\var).
\end{equation}

Indeed, by the theorems of Lusin and Egorov there are an open $U\subset X$ and $\eta\in C(X)$ such that 
$\mu_u(U)<\var/2$, $\eta=\xi$ on $X\setminus U$, and $\xi_k\to\xi$ uniformly on $X\setminus U$. With $k_0$ to be 
determined in a moment,  let $\eta_k=\xi_k$ if $k<k_0$, 
\[
\eta_k=\begin{cases}\eta\quad&\text{on } U\\ \xi_k&\text{on } X\setminus U\end{cases}\qquad\text{if } k\ge k_0.
\] 
Thus $\eta_k\to\eta$ uniformly. Next, $\mu_u(\xi\neq\eta)\le\mu_u(U)<\var$. To estimate 
$\mu_{u_k}(\xi_k\neq\eta_k)$ we pick a compact $K\subset U$ such that $\mu_{u_k}(U\setminus K)<\var/2$ for all $k$ 
and a function $\zeta\in C(X)$ such that
$\chi_K\le\zeta\le\chi_{U}$ (characteristic function). This implies 
$\int_X(\chi_U-\zeta)\,d\mu_{u_k}\le\mu_{u_k}(U\setminus K)<\var/2$ for all $k$. 
Choose
$k_0$ so that
\[
\int_X\zeta\,d\mu_{u_k}<\var/2\qquad\text{if }k\ge k_0.
\]
This is possible, because, for example by \cite[Theorem 3.18]{GZ2}, as $k\to\infty$, the integral above tends to
$\int_X\zeta\,d\mu_u\le\mu_u(U)<\var/2$. Hence for $k\ge k_0$
\begin{equation}  
\mu_{u_k}(\xi_k\neq\eta_k)\le\mu_{u_k}(U)=\int_X(\chi_{U}-\zeta)\,d\mu_{u_k}+\int_X\zeta\,d\mu_{u_k}<\var,
\end{equation}
and the same holds for $k<k_0$ trivially.

We view
$\eta,\eta_k$ as vectors in $T^\infty_u\cE(\omega)$, $T^\infty_{u_k}\cE(\omega)$. According to (3.2)
$\mu_u(\xi\ge\xi^\star(s))=s$. Hence
\[
s-\var<\mu_u\big(\eta\ge\xi^\star(s)\big)< s+\var,
\]
and by (3.3) $\eta^\star(s+\var)\le\xi^\star(s)\le\eta^\star(s-\var)$. The second set of inequalities in (3.6) follows the same
way, using (3.7).

We can apply Lemma 3.4 and conclude that $\eta^\star_k\to\eta^\star$ away from a  countable set. Therefore
\[
\limsup_{k\to\infty}\xi_k^\star(s)\le\lim_{k\to\infty}\eta_k^\star(s-\var)=\eta^\star(s-\var)\le\xi^\star(s-2\var),
\]
for all $s\in[2\var,\mu_0(X)-2\var]$ except countably  many; and similarly 
$\liminf_{k\to\infty}\xi_k^\star(s)\ge \xi^\star(s+2\var)$. If in addition $\xi^\star$ is continuous at $s$, letting $\var\to 0$ 
through a sequence these inequalities prove
$\lim_{k\to\infty}\xi_k^\star(s)=\xi^\star(s)$, which therefore holds for all $s\in[0,\mu_0(X)]$ with countably many exceptions.

From this $\lim_k L(\xi_k)=L(\xi)$ follows the same way as in Lemma 3.4.
\end{proof}

\section{Weak geodesics, $\varepsilon$--geodesics, Jacobi fields}   

If $a<b$ are real numbers, we let
\[S_{ab}=\{s\in\bC: a<\text{Re}s<b\},\]
and denote by $\pi$ the projection $S_{ab}\times X\to X$. 

Following Berndtsson and Darvas \cite[section 3.3]{Be1, D4} we make the following definition.
\begin{defn}  
A path $u:(a,b)\to\psho$ is a subgeodesic if the function $U: S_{ab}\times X\to [-\infty,\infty)$ given by $U(s,x)= u(\text{Re}s)(x)$ is $\pi^\star\omega$--plurisubharmonic.

If $w, w'\in\psho$, the weak geodesic determined by (or connecting) $w, w'$ is $u:(a,b)\to\psho$, 
\begin{equation}   
u=\sup\{v | v:(a,b)\to\psho\text{ is subgeodesic, }\lim_{t\to a}v(t)\le w, \lim_{t\to b}v(t)\le w'\}. 
\end{equation}
\end{defn}
The limits are understood pointwise on $X$; they exist because $\pi^*\omega$--plurisubhar\-monicity
 implies that for each $x\in X$ the function $v(\cdot)(x)$ is convex. It is possible that (4.1) gives $u\equiv -\infty$, not valued 
 in $\psho$; but otherwise the weak geodesic is indeed a path in $\psho$ and is itself a subgeodesic \cite[section 3.1]{D4}. 
 Darvas points out 
that in general  the term ``connecting'' weak geodesic is misleading, as $\lim_a u$ may have little to do with $w$. But, if 
$w,w'$ are bounded, Berndtsson proves by a simple argument that the weak geodesic indeed connects, $\lim_a u=w$,  
$\lim_b u=w'$, uniformly on $X$, \cite[pp. 156-157]{Be1}. If $c<d$, 
the weak geodesic $v: (c,d)\to\psho$ between $w$ and $w'$ is $u$, composed with an affine reparametrization, 
because affine reparametrizations of subgeodesics yield subgeodesics.

In what follows we will only deal with weak geodesics $u$ determined by bounded $w,w'$. Such a 
$u$ is a Lipschitz map into
$B(X)$, and we will refer to its continuous extension to the closed interval $[a,b]$ as a weak geodesic, too.

In (1.1) we defined Mabuchi's connection on $T\cH$ through its parallel transport. A more direct definition takes a smooth 
path $u:[a,b]\to \cH$ and a smooth vector field $\xi:[a,b]\to T\cH$, $\xi(t)\in T_{u(t)}\cH$, along it; the covariant derivative of 
$\xi$ along $u$ is then the vector field $\na_t\xi$ given by
\begin{equation}     
\na_t\xi(t)=\dot\xi(t)-\frac{1}{2}(d_X\dot u(t), d_X\xi(t))_{u(t)}\in C^\infty(X)\approx T_{u(t)}\cH.
\end{equation}
Here $d_X$ is differential on $X$, for fixed $t$, and $(\ \ ,\ \ )_{u(t)}$ is inner product on $T^*X$ induced by the K\"ahler metric of $\omega_{u(t)}$. In (4.2) the left hand side is to be computed for $\xi$ a section of $T\cH$ along $u$; on the right $\xi$ stands for the representation of this section in the canonical trivialization $T\cH\approx \cH\times C^\infty(X)$, so for a function $[a,b]\to C^\infty(X)$; and the equality of the two sides again uses the trivialization of $T\cH$.

Geodesics $u:[a,b]\to \cH$ of $\na$ satisfy $\na_t\dot u(t)=0$. Chen, however, had the idea that the geometry of $\cH$ can be better accessed through $\varepsilon$--geodesics. Define a vector field $F$ on $\cH$ by
\begin{equation}   
F(v)\omega^n_v=\omega^n,\quad v\in \cH, \quad F(v)\in C^\infty(X)\approx T_v \cH.
\end{equation}
If $\varepsilon>0$, an $\varepsilon$--geodesic is a solution $u:[a,b]\to \cH$ of 
\begin{equation}   
\na_t\dot u(t)=\varepsilon F(u(t)), \quad t\in[a,b].
\end{equation}

In what follows, we will just write $d$ for $d_X$. Chen proves \cite{C, Bl1, Bl2}

\begin{thm}   
Given $a<b$ and two potentials
$w,w'\in \cH$, (4.4) has a unique $C^2$ solution  $u=u^\varepsilon: [a,b]\to \cH$ satisfying
$u(a)=w$, $u(b)=w'$. The solution $u$ is smooth, and as an element of $C^\infty([a,b]\to \cH)$, it depends smoothly on 
$w, w'$ (and $a, b,\varepsilon$). Finally, if $w,w'$ are in a fixed compact subset of $\cH$, the forms $dd^cu^\varepsilon(t)$, 
$d\dot u^\varepsilon(t)$ and $\ddot u^\var(t)$ are uniformly bounded on $X$ for $0<\varepsilon<\varepsilon_0$ and $a\le t\le b$.
\end{thm}

It follows by the Arzel\`a--Ascoli theorem and a maximum principle for the Monge--Amp\`ere operator that for fixed $w,w'$
the uniform limit $u=\lim_{\varepsilon\to 0} u^\varepsilon$ exists. The limit $u$ maps into the space
\begin{equation}  
 \cH^{1\overline 1}=\{w\in C(X)\cap\psho: \text{ the current } dd^cw \text{ is bounded}\}, 
 \end{equation}
and the currents $dd^cu(t), d\dot u(t), \ddot u(t)$ are represented by uniformly bounded forms. Also, $u$ is the weak geodesic in the sense of Definition 4.1 to connect $w,w'$.

Consider an $\varepsilon$--geodesic $u:[a,b]\to \cH$. 

\begin{defn}   
A vector field $\xi:[a,b]\to T\cH$ along $u$ is an $\varepsilon$--Jacobi field if there are an interval $I$ containing $0\in\bR$ and a smooth family $I\ni s\mapsto u^s$, each $u^s:[a,b]\to \cH$ an $\varepsilon$--geodesic such that 
$u^0=u$ and $\xi=\partial_su^s |_{s=0}$.
\end{defn}

\begin{lem}   
(a)\  If $I\subset \bR$ is an interval and $v:I\to \cH$ a smooth path, then the covariant derivative of $F$ along $v$ satisfies
\begin{equation}   
\omega^n_{v(s)}\na_s F(v(s))=-n d\big(F(v(s)) d^c\dot v(s)\big)\wedge\omega^{n-1}_{v(s)}.
\end{equation}
(b)\ If $\xi:[a,b]\to T\cH$ is an $\varepsilon$--Jacobi field along an $\varepsilon$--geodesic $u:[a,b]\to \cH$, then
\begin{equation}   
\omega^n_{u(t)}\na^2_t\xi(t)=\frac{1}{4}\big\{\{\dot u(t),\xi(t)\}, \dot u(t)\big\}\omega^n_{u(t)}-
\varepsilon n d \big(F(u(t)) d^c\xi(t)\big)\wedge\omega^{n-1}_{u(t)},
\end{equation}
where $\{\, ,\, \}=\{\,  ,\, \}_{u(t)}$ is Poisson bracket on $T_{u(t)} \cH\approx C^\infty(X)$ for the symplectic form 
$\omega_{u(t)}$.
\end{lem}

Calabi and Chen \cite[Section 2.3]{CC}  derive an equivalent equation for $\varepsilon$--Jacobi fields.

\begin{proof}
(a)\ We will apply (4.2) with $\xi=F\circ v$. Differentiating $F(v(s))\omega^n_{v(s)}=\omega^n$ with respect to $s$ gives
\[
\omega^n_{v(s)}\partial_s F(v(s))=-F(v(s))\partial_s(\omega+dd^c v(s))^n=
-n F(v(s))dd^c\dot v(s)\wedge\omega^{n-1}_{v(s)}.
\]
At the same time
\[
\big(d\dot v(s), dF(v(s))\big)_{v(s)} \omega^n_{v(s)}=2n d F(v(s))\wedge d^c\dot v(s)\wedge\omega^{n-1}_{v(s)},
\]
see e.g. \cite[p.103]{Bl2}. Substituting into (4.2) now gives (4.6).

(b)\ Let $u^s:[a,b]\to \cH$ be a smooth family of $\varepsilon$--geodesics such that $\xi=\partial_s u^s | _{s=0}$, and set $U(s,t)=u^s(t)$. As Mabuchi's connection is torsion free, $\na_s\partial_t U=\na_t\partial_s U$. The curvature of $\na$, evaluated on $\partial_s U(s,t), \partial_t U(s,t)\in T_{U(s,t)}\cH$ is an endomorphism of $T_{U(s,t)} \cH$ that acts on a vector field $\eta(s,t)$ by
\[R(\partial_sU, \partial_t U)\eta=(\na_s\na_t-\na_t\na_s)\eta=\big\{\{\partial_s U, \partial_t U\},\eta\big\} / 4,\]
see \cite[Theorem 4.3]{M}. (Mabuchi's formula does not contain the factor $1/4$, due to different conventions.)

We apply $\na_s$ to the $\varepsilon$--geodesic equation $\na_t\partial_t U(s,t)=\varepsilon F(U(s,t))$, to obtain at $s=0$
\begin{align*}
\varepsilon\na_sF(U)&=\na_s\na_t\partial_t U=R(\partial_s U, \partial_t U)\partial_t U+\na_t\na_s\partial_t U\\
&=(1/4)\big\{\{\partial_s U,\partial_t U\}, \partial_t U\big\} +\na_t\na_t\partial_s U
=(1/4)\big\{\{\xi,\dot u\},\dot u\big\}+\na^2_t\xi.
\end{align*}
Combining (4.6) with this, (4.7) follows.
\end{proof}

\section{The action}    

Consider an invariant convex Lagrangian $L:T\cH\to \bR$. If $u:[a,b]\to \cH$ is a piecewise $C^1$ path, its action is
\begin{equation}    
\cL(u)=\int^b_a L(\dot u(t)) dt.
\end{equation}
Depending on the nature of $L$, this can represent length or energy of a path, but in general it is neither. No mather what $L$, the integral (5.1) is that of a piecewise continuous function by Theorem 2.5, so that it exists as a Riemann integral.

For the purposes of this paper we must consider action for paths beyond $\cH$. The material developed in section 3 allows to define action for paths in the space $B(X)\cap \psho\subset\cE(\omega)$ of bounded $\omega$--plurisubharmonic functions. This is a subset of the Banach space $B(X)$ and, viewing maps into it as maps into $B(X)$, we can talk about various regularity classes of such maps. If $a<b$ are real, the following is easy to check.
\begin{lem}  
A map $u:[a,b]\to B(X)$ is continuous if and only if the functions
\begin{equation}   
 u(\cdot)(x),\quad\text{for } x\in X
\end{equation}
are equicontinuous, and it is $C^k$ for $k=1,2,\dots $ if and only if, in addition, the functions in (5.2) are $k$ times
differentiable, and the $k$'th derivatives are also equicontinuous.
\end{lem}

According to Theorem 3.2, an invariant convex Lagrangian $T\cH\to\bR$ that is strongly continuous on the fibers 
determines a strongly continuous, invariant, convex Lagrangian $L:T^\infty\cE(\omega)\to \bR$. Suppose
$u:[a,b]\to B(X)\cap\psho$ is a $C^1$ path. Since $\omega$--plurisubharmonic functions are quasicontinuous
\cite[Theorem 3.5]{BT}, \cite[Corollary 9.12]{GZ2}, the difference quotients $(u(t)-u(s))/(t-s)$ are quasicontinuous and 
so are their uniform limits
$\dot u(t)$. Hence
 by Lemma 3.4 $L\circ\dot u: [a,b]\to\bR$ is continuous. Clearly
 if $u$ is just piecewise $C^1$, the integral in (5.1) still exists as the integral of a piecewise continuous function, and 
 defines action $\cL(u)$.
 
 If $w,w'\in B(X)\cap \psho$ and $T\in(0,\infty)$, we define the least action, or just action, $\cL_T(w, w')$ between $w,w'$ as
\begin{equation}   
\cL_T(w,w')=\inf_u \cL(u),
\end{equation}
the infimum taken over all piecewise $C^1$ paths $u:[0, T]\to B(X)\cap \psho$ such that $u(0)=w$, $u(T)=w'$.
Note that $w,w'$ can be connected by a smooth path, e.g. $u(t)=(1-t/T)w+(t/T)w'$ connects. We will see that 
$\cL_T(w,w')>-\infty$ (Lemma 9.4). 

Instead of $[0,T]$ if we minimize over paths $[a,a+T]\to B(X)\cap\psho$, the infimum in (5.3) does not change. However, 
in general $\cL_T(w,w')$ will depend on $T$; it will not if $L$ is positively homogeneous. In general 
\[\cL_T(w,w')+\cL_S(w',w'')\ge\cL_{T+S}(w,w'')\]
follows by concatenating paths. Of course, $\cL_T(w,w')=\cL_T(w', w)$ should be expected only if $L$ is even, $L(-\xi)=L(\xi)$.

In the two results below, $L:T^\infty\cE(\omega)\to\bR$ is a strongly continuous, invariant, convex Lagragian.

\begin{thm}[Principle of least action]    
If a $C^1$ path $v:[0,T]\to B(X)\cap\psho$ is a weak geodesic, then $\cL(v)=\cL_T(v(0), v(T))$.  
\end{thm}

It is not hard to show that piecewise $C^1$ geodesics in $B(X)\cap\psho$ are automatically $C^1$, and by Chen's work,
Theorem 4.2, weak geodesics connecting points in $\cH$  are also $C^1$. Nonetheless, by Example 5.4 below,
general weak geodesics in 
$B(X)\cap\psho$ are not even left or right differentiable, and it is dubious if action along such paths can be defined by 
an integral.

\begin{thm}    
If $u,v:[a,b]\to B(X)\cap \psho$ are weak geodesics, then for $S>0$
the function
\[ 
[a,b]\ni t\mapsto \cL_S(u(t),v(t))\in\bR
\]
is convex.
\end{thm}

If $L$ is absolutely homogeneous, $L(c\xi)=|c| L(\xi)$, and vanishes only on zero vectors, then
action $\cL_S$ is distance measured in a Finsler metric and is independent of $S$; the statement of Theorem 5.3 is an indication of seminegative curvature.

The proofs will take sections 6--9. First we prove approximate versions, with $\varepsilon$--geodesics in $\cH$ replacing 
weak geodesics. The approximate versions depend on two facts. First, that $L$ is convex along $\varepsilon$--Jacobi 
fields; second, as a consequence, 
a triangle type inequality holds for triangles in $\cH$ with two sides $\varepsilon$--geodesics (Theorem 6.1, Lemma 7.2). It 
is a technical point but noteworthy that the approximate results contain no error term, no $O(\varepsilon)$. By letting 
$\varepsilon\to 0$ we obtain a principle of least action in $\cH^{1\bar 1}$ (Corollary 7.4). Approximating weak geodesics in 
$B(X)\cap\psho$ by weak geodesics in $\cH^{1\bar 1}$ we obtain the same in $B(X)\cap \psho$ in section 8. Theorem 5.3 is proved by the same approximation scheme in section 9.

To conclude this section we discuss smoothness of weak geodesics $v:[a,b]\to B(X)\cap\psho$. For fixed $x\in X$, the
function  $v(\cdot)(x)$ is convex, hence has left and right derivatives at every $t\in(a,b)$, that we will denote 
$\partial^\mp_t v(t)(x)$. These one sided derivatives are bounded, since $v$ is Lipschitz according to Berndtsson 
\cite[section 2.2]{Be1}, see Lemma 9.3. Let $(X,\omega)$ be complex projective space $\bP_n$ with 
the Fubini--Study metric.

\begin{example} 
There is a weak geodesic $v:[a,b]\to B(X)\cap\psho$ that is not left or right differentiable (as a map into
$B(X)$) at any $t\in[a,b]$. Furthermore, for each $t\in(a,b)$ the derivatives 
$\partial^-_tv(t)$ and $\partial^+_tv(t)$, computed pointwise on $X$, differ $\mu_{v(t)}$ almost everywhere. 
\end{example}

Our $v$ will be a toric weak geodesic. Such weak geodesics can be obtained using an extension of Guan's
correspondence between toric geodesics and linear geodesics in a  suitable Hilbert space $L^2(P)$, 
see \cite{G}. Nevertheless, since the correspondence 
between properties of toric weak geodesics and linear geodesics is less straightforward than for geodesics
in $\cH$, we will just write down one possible example and verify its character directly. 

We view $\bP_n$ as $\bC^n\cup\bP_{n-1}$, and denote by $|\ |$ Euclidean norm on $\bC^n$. 
Keeping in mind that  $\omega=dd^c\log(1+|x|^2)$ on $\bC^n$, we
define $v:[a,b]\to C(X)\cap\psho$ by
\begin{equation} 
v(t)(x)=\begin{cases}2\max(t,\log|x|)-\log(1+|x|^2),\quad&\text{if } t\in[a,b],\, x\in\bC^n\\
0,&\text{if }t\in[a,b],\,x\in\bP_{n-1}.\end{cases}
\end{equation}
In fact $v(t)$ is smooth in a neighborhood of $\bP_{n-1}$. To check that $v$ is a weak geodesic, let
$S_{ab}=\{s\in\bC:a<\text{Re}s<b\}$ as in section 
4 and
\[
V(s,x)=\begin{cases}2\max(\text{Re}s,\log|x|)-\log(1+|x|^2), \quad&\text{if } s\in \bar S_{ab}, \,x\in\bC^n\\
0, &\text{if }s\in \bar S_{ab},\, x\in\bP_{n-1}.\end{cases}
\]
Thus $V$ is $\pi^*\omega$--plurisubharmonic on $S_{ab}\times X$, i.e., $v$ is a subgeodesic. 
Suppose $u:[a,b]\to B(X)\cap\psho$ is the weak geodesic connecting $v(a), v(b)$, 
so that $v\le u$. Let $U(s,x)=u(\text{Re} s)(x)$ for $s\in \bar S_{ab}$, $x\in X$. Since
$u(a)=v(a)$ and $u(b)=v(b)$ are continuous, $U$ is continuous at points in $\partial S_{ab}\times X$. If $y\in X$,
consider the map 
\[
\phi_y:\bar S_{ab}\ni s\mapsto (s,e^sy)\in \bar S_{ab}\times X,
\] 
with the understanding that $e^sy=y$ if $y\in\bP_{n-1}$. The pull back $\phi_y^*V$ is smooth, and 
$dd^c\phi_y^*V=2dd^c(\text{Re}s+\max(0,\log|y|))-\phi_y^*\omega=-\phi^*_y\omega$. Hence the bounded function
$\phi_y^*(U-V)$ is subharmonic on $S_{ab}$, and by the maximum principle it is $\le 0$. As $y$ ranges over $X$,
the strips $\phi_y(S_{ab})$ foliate $X$. Therefore $U\le V$. Along with $v\le u$ this implies that 
$v=u$ is indeed a weak geodesic.

For fixed $x\in X$ one computes from (5.4) the one sided derivatives of $v(\cdot)(x)$,
\begin{equation} 
\partial^-_t v(t)(x)=\begin{cases} 2&\text{if } t>\log|x|\\ 0&\text{otherwise},\end{cases}\qquad
\partial^+_t v(t)(x)=\begin{cases} 2&\text{if } t\ge\log|x|\\ 0&\text{otherwise}.\end{cases}
\end{equation}
On the one hand, for each $t$ these are discontinuous functions on $X$; on the other, each difference quotient of $v$ is
continuous. Hence $\partial^\pm_tv(t)$ cannot be the limit, in $B(X)$, of difference quotients, and so $v$ as a map into 
$B(X)$ has no one sided derivatives.

Finally, for each $t$ (5.4) gives that the measure induced by $\omega_{v(t)}^n$ is invariant under U$(n)$ rotations of $X$ 
and is supported  on the sphere $|x|=e^t$. Hence it is a nonzero multiple of area measure on that sphere; by (5.5) therefore 
$\mu_{v(t)}$--almost everywhere $\partial^-_tv(t)\neq \partial^+_tv(t)$.

The example shows that action even along a weak geodesics in $B(X)\cap\psho$ cannot be defined by the integral (5.1), since 
the integral depends on whether $\dot u(t)$ is interpreted as left or right derivative. Perhaps the correct 
interpretation is the average of the two; this is what Theorem 10.1 and Lemma 10.2 seem to suggest.

\section{Divergence of $\varepsilon$--Jacobi fields}   

In this section we stay in $\cH$, and consider invariant convex Lagrangians $L:T\cH\to\bR$, that are just continuous. 
Recall that given $\varepsilon>0$, an $\varepsilon$--geodesic $u:[a,b]\to \cH$ satisfies the equation
\begin{equation}    
  \na_t\dot u(t)=\varepsilon F(u(t)),
\end{equation}
where the vector field $F:\cH\to T\cH$ is defined by $F(v)\omega^n_v=\omega^n$. Infinitesimal variations of $\varepsilon$--geodesics are $\varepsilon$--Jacobi fields. If $\xi:[a,b]\to T\cH$ is an $\varepsilon$--Jacobi field along an $\varepsilon$--geodesic $u:[a,b]\to T\cH$, by Lemma 4.4(b)
\begin{equation}   
\omega^n_{u(t)}\na^2_t\xi(t)=\frac{1}{4}\big\{\{\dot u(t),\xi(t)\}, \dot u(t)\big\} \omega^n_{u(t)}- 
\varepsilon n d \big(F(u(t)) d^c\xi(t)\big)\wedge \omega^{n-1}_{u(t)}.
\end{equation}
All our subsequent results rest on the following theorem.

\begin{thm}   
If $\xi:[a,b]\to T\cH$ is an $\varepsilon$--Jacobi field along an $\varepsilon$--geodesic $u:[a,b]\to \cH$, then $L\circ\xi$ is
a convex function on $[a,b]$.
\end{thm} 
This will be derived from a special case.

\begin{lem}   
Given $u_0\in \cH$ and $f_0\in B(X)$, Theorem 6.1 holds for the Lagrangian 
\begin{equation}    
L(\eta)=\sup_{(f,v)\sim(f_0, u_0)}\ \int_X f\eta d\mu_v,\quad \eta\in T_v \cH,
\end{equation}
cf. (2.1).
\end{lem}

To prove Lemma 6.2 we need some preparation. If $Y$ is any set, we say that functions $g,h:Y\to\bR$ are similarly ordered if $(g(x)-g(y))(h(x)-h(y))\ge 0$ for all $x,y\in Y$. Equivalently, $g(x)< g(y)$ should imply $h(x)\le h(y)$. The relation is not transitive, any function is similarly ordered as a constant.

\begin{lem}   
Let $Y$ be an oriented smooth manifold and $a_{ij}$ smooth functions, $V_i$ smooth vector fields on it, $i,j=1,\dots, k$. Assume the matrix $(a_{ij})$ is symmetric and positive semidefinite everywhere. If $g\in C^\infty(Y)$ and a locally integrable
$h:Y\to\bR$ are similarly ordered, then the current $Q=\sum_{i,j} a_{ij} (V_i g)(V_j h)\ge 0$.
\end{lem}

\begin{proof}
It suffices to prove when $h$ is bounded (in general we replace $h$ by $h_R=\min\big(R,\max(-R,h)\big)$ and let $R\to\infty$).
Assume first that in addition there is a smooth increasing $H:\bR\to \bR$ such that $h=H\circ g$. Then $Q=H'(g)\sum{a_{ij}}(V_ig)(V_j g)\ge 0$. The same follows if $H$ is any increasing function, by writing it as $\lim_p\lim_q H_{pq}$ (pointwise limit), with locally uniformly bounded smooth increasing $H_{pq}$.

Now consider general $g,h$. Let $I$ denote the range of $g$, and for $t\in I$ define
\[m(t)=\inf\{h(x): g(x)=t\}, \quad M(t)=\sup\{h(x): g(x)=t\}.\]
If $g(x)=t< g(y)=\tau$, then $h(x)\le h(y)$, which means that
\[m(t)\le M(t)\le m(\tau)\le M(\tau) \quad \text{when } t<\tau.\]
In particular, $m$ and $M$ are increasing functions, and coincide on int$\, I$ wherever one of them is continuous, that is, apart from a countable set $T\subset I$. On $g^{-1}(I\setminus T)$ we have $h=m\circ g$. If $t$ is a  regular value of $g$, then $g^{-1}(t)$ has measure 0. Hence on the regular set of $g$ the functions $h$ and $m\circ g$ agree a.e., and the induced currents simply agree there. By what we already proved, $Q\ge 0$ on the set where $dg\ne 0$. We still need to understand the contribution of the critical set $C=(dg=0)$.

Let $\chi:[0,\infty)\to [0,1]$ be a smooth function, $\chi(t)=0$ if $t\le 1$, $\chi(t)=1$ if $t\ge 2$. Endow $Y$ with a Riemannian metric and denote by dist$(\cdot,C)$ distance to $C$. This is a Lipschitz function with Lipschitz constant 1. For $s>0$, the function $\chi(s \text{ dist}(\cdot,C))$ has Lipschitz constant $O(s)$; it vanishes in the $1/s$ neighborhood of $C$ and equals 1 outside the $2/s$ neighborhood. Let $\rho_s\in C^\infty(Y)$ have the same properties. To prove the lemma we need to show that if $\theta\ge 0$ is a compactly supported smooth volume form on $Y$, then 
\[
0\le\int_YQ\theta=-\int_Y h\sum_{i,j}\pounds_j(\theta{a_{ij}}V_i g),
\]
where $\pounds_j$ stands for Lie derivative along $V_j$.

The inequality holds if $\theta$ is replaced by $\theta\rho_s$, because $Q\ge 0$ in a neighborhood of supp $\theta\rho_s$. The point is that the functions $\pounds_j(\theta\rho_s a_{ij} V_i g)$ are uniformly bounded and tend to 
$\pounds_j(\theta a_{ij} V_i g)$ a.e. as $s\to\infty$. The former
is verified by applying Leibniz rule to the products, and checking each term. The only term that needs speaking for is
$\theta a_{ij}(V_ig)(V_j\rho_s)$. But since $|V_ig|\,|V_j\rho_s|$ attains its maximum on 
$\{y\in Y: 1/s\le \text{ dist}(y, C)\le 2/s\}$, 
this maximum is $O(1/s) O(s)=O(1)$. As to convergence,
\begin{equation*}
\lim_{s\to\infty} \pounds_j(\theta\rho_s a_{ij} V_i g)=\begin{cases}\pounds_j(\theta a_{ij} V_i g) &\text{ on } Y\setminus C\\
0 &\text{ on } C.\end{cases}
\end{equation*}
If $\pounds_j(\theta a_{ij} V_i g)(y)\neq 0$ at some $y\in C$, then $(V_jV_ig)(y)\neq 0$; thus $y$ is a zero and a regular point
of $V_ig$. Such points form a hypersurface in $Y$, of zero measure. This proves the a.e. convergence statement.

We conclude by dominated convergence:
\[
\int_Y Q\theta=-\lim_{s\to\infty}\int_Y h\sum_{i,j}\pounds_j(\theta\rho_s a_{ij}V_ig)=
\lim_{s\to\infty}\int_Y Q\theta\rho_s\ge 0.
\]
\end{proof}

\begin{proof}[Proof of Lemma 6.2]
The plan is to construct for every $t_0\in(a,b)$ a 
family $f(t)\in B(X)$ such that $(f(t), u(t))\sim(f_0,u_0)$ and $A(t)=\int_X f(t)\xi(t) d\mu_{u(t)}\le L(\xi(t))$ satisfies
\[A(t_0)=L(\xi(t_0)),\qquad \ddot A(t_0)\ge 0.\]
To simplify notation we can assume $t_0=0$. At the price of replacing $f_0$ by $f_1$ such that $(f_0, u_0)\sim(f_1, u(0))$, we can assume $u(0)=u_0$. Further to simplify we can arrange that $f=f_0$ realizes
\begin{equation*}
\sup_{(f,u(0))\sim (f_0, u(0))}\int_X f\xi(0) \,d\mu_{u(0)}, \quad\text{ i.e., }\quad L(\xi(0))=\int_xf_0\xi(0)\,d\mu_{u(0)} \,;
\end{equation*}
this is possible simply because the supremum is attained, see e.g., \cite[Lemma 6.2]{L4}. The same lemma says that there is a maximizing $f$ that is similarly ordered as $\xi(0)$, and accordingly we will work with $f_0$ similarly ordered as 
$\xi(0)$. 

For a moment suppose $u:[a,b]\to \cH$ is an arbitrary smooth path, and parallel transport $T_{u(0)} \cH\to T_{u(t)} \cH$ along $u$ is given by pull back by a symplectomorphism $\varphi(t):(X,\omega_{u(t)})\to(X,\omega_{u(0)})$. Suppose $\eta:[a,b]\to T_{u(0)}\cH$ is smooth; then $t\mapsto \eta(t)\circ \varphi(t)$ defines a vector field along $u$. Parallel transport 
intertwines differentiation and covariant differentiation:
\[\na_t(\eta(t)\circ\varphi(t))=\dot\eta(t)\circ\varphi(t) \quad\text{ and }\quad
\na_t^2(\eta(t)\circ\varphi(t))=\ddot\eta(t)\circ\varphi(t).\]

When $u$ is an $\varepsilon$--geodesic and $\xi$ an $\varepsilon$--Jacobi field along it, as in the lemma, choose $\eta$ so that $\eta(t)\circ\varphi(t)=\xi(t)$. By (6.2), at $t=0$, 
\begin{equation}    
\ddot\eta(0)\omega^n_{u(0)}=(1/4)\{\{\dot u(0),\eta(0)\},\dot u(0)\}\omega^n_{u(0)}-
\varepsilon n d\big(F(u(0)) d^c\eta(0)\big)\wedge\omega^{n-1}_{u(0)}.
\end{equation}
With $f(t)=f_0\circ \varphi(t)$ we let
\[
A(t)=\int_X f_0\eta(t)\omega^n_{u(0))}=\int_X f(t)\xi(t)\omega^n_{u(t)}\le L(\xi(t)),
\]
then $\ddot A(t)=\int_X f_0\ddot\eta(t)\omega^n_{u(0)}$. In view of (6.4)
\begin{align*}
\ddot A(0)&=\frac{1}{4}\int_X f_0\{\{\dot u(0), \eta(0)\}, \dot u(0)\}\omega^n_{u(0)}-
\varepsilon n\int_X f_0 d\big(F(u(0))d^c\eta(0)\big)\wedge \omega^{n-1}_{u(0)}\\
&=\frac{1}{4}\int_X\{\dot u(0), f_0\}\{\dot u(0), \eta(0)\}\omega^n_{u(0)}+\varepsilon n\int_X F(u(0)) d f_0\wedge d^c\eta(0)\wedge\omega^{n-1}_{u(0)}.
\end{align*}
In the last line $\{\dot u(0),f_0\}$ and $df_0$ are currents. By Lemma 6.3 the first integrand in this last line is $\ge 0$, since $f_0$ and $\eta(0)=\xi(0)$ are similarly ordered; and so is, for the same reason, $2n df_0\wedge d^c\eta (0)\wedge \omega^{n-1}_{u(0)}=(df_0, d\eta(0))_{u(0)}\omega^n_{u(0)}
$, cf. \cite[p.103]{Bl2}.

To summarize, we have shown that for every $t_0\in(a,b)$ there is a function $A\in C^\infty[a,b]$ such that
\[A(t)\le L(\xi(t)), \text{ with equality when } t=t_0, \text{ and } \ddot A(t_0)\ge 0.\]
By a standard argument this implies that $L\circ\xi$ is convex. First one notes that if $p>0$ and $q\in\bR$, the function $L(\xi(t))+pt^2+qt$ cannot have a local maximum at any $t_0\in (a,b)$, because with the $A$ we have constructed 
$A(t)+pt^2+qt$ has no local maximum at $t_0$. It follows that on any subinterval $[\alpha,\beta]\subset[a,b]$, $L(\xi(t))+pt^2+qt$ attains its maximum at one of the endpoints, whence $L(\xi(t))+pt^2$ is convex. Letting $p\to 0$ we see that $L\circ\xi$ itself is also convex. 
\end{proof}

\begin{proof}[Proof of Theorem 6.1]
Clearly, Lemma 6.2 implies that if $a\in\bR$, $g\in B(X)$, and 
\[L_{a,g}(\eta)=a+\sup_{(f,v)\sim (g,u_0)} \int_X f\eta d\mu_v,\quad \eta\in T_v \cH,\]
then $L_{a,g}$ is convex along any $\varepsilon$--Jacobi field. Since by Theorem 2.4 a general invariant convex Lagrangian is the supremum of a family of such $L_{a,g}$, the theorem follows.
\end{proof}

\section{Least action in $\cH$ and $\cH^{1\bar 1}$}    

In this section we will compare the actions along weak geodesics in $\cH^{1\bar 1}$ and along
general paths in $\cH$. Recall the notation $T^c\cH=\cH\times C(X)$.

\begin{thm}    
Suppose a Lagrangian $L:T^c\cH\to\bR$ is invariant and convex. Consider a piecewise $C^1$ path 
$u:[0,T]\to \cH$ and a weak geodesic $v:[0,T]\to \cH^{1\bar 1}$. If $u(0)=v(0)$ and $u(T)=v(T)$, then
\begin{equation}    
\frac 1T\int^T_0 L\circ\dot u\ge L(\dot v(0)).
\end{equation}
\end{thm}

First we prove a variant.

\begin{lem}  
Suppose an invariant convex Lagrangian $L:T\cH\to\bR$ is positively homogeneous, $L(c\xi)=c L(\xi)$ if $c>0$. Consider a triangle in $\cH$ formed by a piecewise $C^1$ path $u:[a,b]\to \cH$ and $\varepsilon$--geodesics $v_a,v_b:[0,T]\to \cH$; so that $v_a(0)=v_b(0)$ and $v_a(T)=u(a)$, $v_b(T)=u(b)$. Then 
\begin{equation}     
\frac 1T\int^b_a L\circ \dot u\ge L(\dot v_b(0)-\dot v_a(0)).
\end{equation}
\end{lem}
Note that positive homogeneity implies the triangle inequality $L(\xi+\eta)\le L(\xi)+L(\eta)$ for $w\in\cH$ and
$\xi,\eta\in T_w\cH$.
\begin{proof}
Because of the additive nature of (7.2), we can assume $u$ is $C^1$, not only piecewise, and then by simple approximation that it is even $C^\infty$. For each $s\in[a,b]$ let $U(s,\cdot): [0,T]\to \cH$ denote the $\varepsilon$--geodesic connecting $v_a(0)=v_b(0)$ with $u(s)$. According to Theorem 4.2, that is, by Chen's work, there is a unique such geodesic, and $U\in C^\infty([a,b]\times [0, T])$. Thus $\xi^s=\partial_s U(s,\cdot)$ is an $\varepsilon$--Jacobi field and $\xi^s(0)=0$. By Theorem 6.1 $L\circ\xi^s$ is convex on $[0,T]$. Using $\partial_t$ (and later, dot) to denote right derivative, therefore
\[L(\xi^s(T))\ge L(\xi^s(0))+T\partial_t |_{t=0} L(\xi^s(t)).\]
By homogeneity, the first term on the right is 0. To compute the second, let $\eta(t)\in T_{U(0,0)} \cH$ denote the parallel translate of $\xi^s(t)\in T_{U(s,t)} \cH$ along $U(s,\cdot)$. Thus
\begin{align*}
\lim_{t\to 0} L(\xi^s(t))/t&=\lim_{t\to 0} L(\eta(t))/t=\lim_{t\to 0}L(\eta(t)/t)\\
&=L(\na_t |_{t=0}\xi^s(t))=L(\na_t |_{t=0}\partial_s U(s,t))=L(\partial_s\partial_t |_{t=0} U(s,t)).
\end{align*}
The last equality is because $\na$ has no torsion, and $U(\cdot,0)$ is constant. Hence, using Jensen's inequality as well, 
\begin{align*}
\frac 1T\int^b_a L(\partial_s u(s))ds=\frac 1T\int^b_a &L(\xi^s(T)) ds\ge\int^b_a L(\partial_s\partial_t |_{t=0} U(s,t))ds\\
\ge L\Big(\int^b_a\partial_s\partial_t& |_{t=0} U(s,t)ds\Big)=L(\dot v_b(0)-\dot v_a(0))
,\end{align*}
as claimed.
\end{proof}

\begin{proof}[Proof of Theorem 7.1]
For $\varepsilon>0$ let $v^\varepsilon,w^\var:[0, T]\to \cH$ denote the $\varepsilon$--geodesics connecting $u(0)$ with 
$u(T)$, respectively, $u(0)$ with itself. Again by Chen \cite{C}, see also B\l ocki \cite{Bl1, Bl2}, $v^\varepsilon\to v$ and
$w^\var\to u(0)$  in such a way that $\dot v^\varepsilon(0)\to\dot v(0)$ and $\dot w^\var(0)\to 0$ in $T^c_{u(0)}\cH$ as 
$\varepsilon\to 0$. Suppose first that $L$ is positively homogeneous, and apply Lemma 7.2 with $[a,b]=[0,T]$, 
$v_a=w^\var$, $v_b=v^\varepsilon$. We obtain
\[
\frac 1T\int^T_0 L\circ\dot u\ge L(\dot v^\varepsilon(0)-\dot w^\var(0)).
\]
Letting $\var\to 0$,
\begin{equation} 
\frac 1T\int^T_0 L\circ\dot u\ge L(\dot v(0))
\end{equation}
follows.
This is true even if $L$ is not positively homogeneous but $L$ plus a constant is.  Since a general $L$ is the supremum of Lagrangians of form positively homogeneous plus constant, see Theorem 2.4, (7.3) follows for a general $L$.
\end{proof}

\begin{lem}    
If $L:T^\infty\cE(\omega)\to\bR$ is strict rearrangement invariant, and
$v:[0,T]\to \cH^{1\bar 1}$ is a weak geodesic, then $L\circ \dot v$ is constant. Hence
\[L(\dot v(0))= \frac 1T \int^T_0 L\circ\dot v.\]
\end{lem}
This can be seen as an instance of Noether's theorem on conserved quantities, albeit in an unusual setting.

\begin{proof}
Berndtsson \cite[Proposition 2.2]{Be2} discovered that $\dot v(t)\in B(X,\mu_{v(t)})$ are equidistributed for all $t$, although he worked with integral K\"ahler classes $[\omega]$ and weak geodesics terminating in $\cH$ only. At any rate, \cite[Lemma 4.10]{D2} implies the general result. Since $L$ is invariant, the lemma follows.
\end{proof}

Together with Theorem 7.1 this almost proves the principle of least action in $\cH^{1\bar 1}$:

\begin{cor}   
Suppose $L:T^\infty\cE(\omega)\to\bR$ is a strongly continuous, invariant, and convex Lagrangian.
If $u:[0,T]\to \cH$ is a piecewise $C^1$ path and $v:[0,T]\to \cH^{1\bar 1}$ is a weak geodesic between the same endpoints, then $\int^T_0 L\circ\dot u\ge \int^T_0 L\circ\dot v$. 
\end{cor}

\section{Least action in $B(X)\cap \mathrm{PSH}(\omega)$}    

Here we will extend Corollary 7.4 to $u,v$ taking values in $B(X)\cap\psho$ (Theorem 5.2). In this section 
$L:T^\infty\cE(\omega)\to\bR$ is assumed to be strongly continuous, invariant, and convex.

\begin{thm}    
Suppose $u,v:[0, T]\to B(X)\cap\psho$ have the same endpoints: $u(0)=v(0)$, $u(T)=v(T)$. If $u$ is piecewise $C^1$ and $v$ is a $C^1$ weak geodesic, then $\int^T_0 L\circ \dot u\ge\int^T_0 L\circ\dot v.$
\end{thm}

This will be derived from Corollary 7.4 by approximation.

\begin{lem}    
Suppose $u:[0,T]\to B(X)\cap \psho$ is a piecewise $C^1$ path, and $w_j,w'_j\in \cH$ decrease to $u(0)$, respectively,
$u(T)$, as $j\to\infty$. Then there are a sequence $J\subset\bN$ 
and for $j\in J$ piecewise linear $u_j:[0,T]\to \cH$ such that 
$u_j(0)=w_{j}$, $u_j(T)=w_{j}'$, and $\int_0^TL\circ \dot u_j\to\int_0^TL\circ\dot u$ as $J\ni j\to\infty$.
\end{lem}

As said, at points where $u, u_j$ are not differentiable, $\dot u,\dot u_j$ mean right derivatives.

\begin{proof}
Choose $t_0=0<t_1<\dots<t_p=T$ so that $u$ is $C^1$ on each $[t_{i-1}, t_i]$. 
Suppose first that $u$ is even linear on $[t_{i-1}, t_i]$. In this case $J$ will be all of $\bN$. A simple special instance of 
regularization, see \cite{De, DP} and especially \cite{BK}, provides $z_{ij}\in \cH$ such that 
$z_{ij}$ decreases to $u(t_i)$ as $j\to\infty$ for $i=0,\dots,p$. We take $z_{0j}=w_j$ and $z_{pj}=w'_j$, and arrange that 
the $z_{ij}$ are uniformly bounded. Linearly interpolating on $[t_{i-1}, t_i]$ between $z_{i-1,j}$ and $z_{ij}$ we obtain the functions 
$u_j$ sought. Indeed, $u_j(t)$ decreases to $u(t)$, and
\[
\dot u_j(t)=\frac{z_{ij}-z_{i-1,j}}{t_i-t_{i-1}}\in T_{u_j(t)}\cH,\qquad\text{when } t\in [t_{i-1},t_i),
\]
are uniformly bounded and  tend to $\dot u(t)$ as $j\to\infty$. Since 
$\omega$--plurisubharmonic functions are quasicontinuous  \cite[Corollary 9.12]{GZ2}, so are the difference quotients
$\dot u(t)$. According to 
\cite[Proposition 9.11]{GZ2} this 
implies convergence in capacity, whence $\lim_jL(\dot u_j(t))=L(\dot u(t))$ by Lemma 3.4. Since $\dot u_j(t)$ are uniformly 
bounded, so are $L(\dot u_j(t))$ by equi--Lipschitz  continuity, Lemma 3.3. The dominated convergence 
theorem gives therefore $\lim_j\int_0^TL\circ \dot u_j=\int_0^TL\circ\dot u$.

For general $u$, partition each $[t_{i-1}, t_i]$ into $k$ equal parts. Construct $v_k:[0,T]\to B(X)\cap\psho$ that agrees with 
$u$ at each partition point, and is linear in between. Then $v_k\to u$ and $\dot v_k\to\dot u$ uniformly, because $\dot u$ is 
uniformly continuous on $[t_{i-1}, t_i)$. Hence $L\circ\dot v_k\to L\circ\dot u$ by Lemma 3.4 and, again by dominated
convergence, $\int_0^TL\circ\dot v_k\to\int_0^T L\circ\dot u$.
By what we have already proved, for each $k$ we can find $j=j_k>j_{k-1}$ and piecewise linear $u_j:[0,T]\to \cH$ 
such that $u_j(0)=w_j$, $u_j(T)=w'_j$, and
 \[
 \Big|\int_0^TL\circ\dot u_j-\int_0^TL\circ\dot v_k\Big|<\frac1k.
 \]
 Thus $J=\{j_1,j_2,\dots\}$ will do.

\end{proof}

\begin{lem}   
Let $v,v_j: [a, b]\to \psho$ be weak geodesics. If $v_j(t)$ decreases to $v(t)$ when $t=a, b$, then $v_j(t)$ decreases to $v(t)$ for all $t\in [a,b]$.
\end{lem}

This is \cite[Proposition 3.15]{D4}.---There is one more ingredient that goes into the proof of Theorem 8.1. 

\begin{lem}    
Consider a weak geodesic $v:[0,T]\to B(X)\cap\psho$. If it is right differentiable at $t\in[0,T)$, then the right derivative 
$\dot v(t)$ is 
quasicontinuous. Moreover, $L\circ\dot v$ is constant on the subset $D\subset(0,T)$ where $v$ is differentiable. Finally,
if $v_j:[0,T]\to \cH^{1\bar1}$ are weak geodesics that decrease to $v$, then $L\circ\dot v_j\to L\circ\dot v$ on $D$. 
\end{lem}

\begin{proof}
As said, plurisubharmonic functions are quasicontinuous, hence so are the difference
quotients $(v(t+s)-v(t))/s$, and their uniform limit, $\dot v(t)$. Next we turn to the last statement, that we reduce to 
Lemma 3.4. First we fix $t\in D$ and show that $\dot v_j(t)\to\dot v(t)$ in sup norm $\|\ \|$. If $\varepsilon>0$, 
there is an $s>0$ such that
\[\Big\|\dot v(t)-\frac{v(t\pm s)-v(t)}{\pm s}\Big\| <\varepsilon,\]
and so there is a $j_0$ such that for $j>j_0$
\[\Big\|\dot v(t)-\frac{v_j(t\pm s)-v_j(t)}{\pm s}\Big\| <\varepsilon.\]
Convexity implies
\[
\frac{v_j(t- s)-v_j(t)}{- s}\le\dot v_j(t)\le\frac{v_j(t+s)-v_j(t)}{s},
\]
whence $\|\dot v_j(t)-\dot v(t)\|<\varepsilon$.

Given that $v_j(t)$ decreases to $v(t)$, that $\dot v_j(t)\to\dot v(t)$ in $B(X)$,  and that $\dot v(t)$ is quasicontinuous,
$t\in D$, Lemma 3.4 implies $L\circ\dot v_j\to L\circ\dot v$ on $D$. 

To prove the second statement, construct $w_j,w'_j\in \cH$ that decrease to $v(0)$, $v(T)$, and let 
$v_j:[0,T]\to \cH^{1\bar 1}$ 
be the weak geodesic that joins them. By Lemma 8.3 $v_j$ decreases to $v$ and by Lemma 7.3 $L\circ\dot v_j$ is 
constant. According to what we just proved, $L\circ\dot v_j\to L\circ\dot v$ on $D$, and $L\circ \dot v$ must be 
constant there.
\end{proof}

In particular, if $v:[a,b]\to B(X)\cap\psho$ is a weak geodesic of class $C^1$, then $L\circ\dot v$ is constant on $(a,b)$.
Using this with different Lagrangians one can show that in fact $\dot v(t)\in B(X,\mu_{v(t)})$ are equidistributed for 
$a<t<b$. Darvas points 
out in \cite[p. 1305]{D3} that for general weak geodesics even in $C(X)\cap\psho$ this is no longer true for $t=a,b$. 

\begin{proof}[Proof of Theorem 8.1] Construct $w_j,w'_j\in\cH$ decreasing to $u(0),u(T)$, and
let $u_j:[0,T]\to \cH$, $j\in J$, be as in Lemma 8.2.
Let $v_j:[0,T]\to \cH^{1\bar 1}$ be the weak geodesic connecting $w_j$ and $w'_j$, $j\in J$. 
By Corollary 7.4
\begin{equation}    
\int^T_0 L\circ\dot u_j\ge \int^T_0 L\circ\dot v_j.
\end{equation}
The integral on the left tends to  $\int_0^TL\circ\dot u$ as $j\to\infty$. The integrand on the right is constant for each $j$, and 
on $(0,T)$ converges unformly to $L\circ\dot v$ by Lemma 8.4. Hence $\lim_{J\ni j\to\infty}\int^T_0 L\circ\dot v_j=\int^T_0 L\circ \dot v$ and letting $j\to \infty$ in (8.1) we obtain the theorem.
\end{proof}

\section{Convexity of the action}     

In this section the Lagrangian $L:T^\infty\cE(\omega)\to\bR$ is strongly continuous, invariant, and convex. We first investigate the least action, cf. (5.3), between two $\varepsilon$--geodesics, and then by letting $\varepsilon\to 0$ we prove Theorem 5.3, which was:
\begin{thm}    
If $u,v:[a,b]\to B(X)\cap\psho$ are weak geodesics, then for any $S\in(0,\infty)$ the function $\cL_S(u,v):[a,b]\to \bR$ is convex.
\end{thm}

The $\varepsilon$--variant is as follows:

\begin{lem}    
If $u,v:[a,b]\to\cH$ are $\varepsilon$--geodesics, then for any $S\in (0,\infty)$ the function $\cL_S(u,v):[a,b]\to\bR$ is 
convex.
\end{lem}

\begin{proof}
Let $a\le\alpha<\beta\le b$. Suppose $U:[0,S]\times [\alpha,\beta]\to \cH$ is a smooth map such that $U(s,\cdot)$ is an $\varepsilon$--geodesic for all $s$, and $U(0,\cdot)=u$, $U(S,\cdot)=v$. Hence $\xi^s=\partial_s U(s,\cdot)$ is an $\varepsilon$--Jacobi field, $0\le s\le S$, and by Theorem 6.1 $L\circ\xi^s$ is convex. Therefore, with $0\le\lambda\le 1$ and $t_\lambda=(1-\lambda)\alpha+\lambda\beta$

\begin{equation}   
\cL(U(\cdot,t_\lambda))=\int^S_0 L(\xi^s(t_\lambda))\,ds\le (1-\lambda)\int^S_0 L(\xi^s(\alpha))\,ds +\lambda\int^S_0 L(\xi^s(\beta))\,ds.
\end{equation}
Fix $\delta>0$. Given $u,v$, we can choose $U$ (uniquely) so that both $w^\delta_\alpha=U(\cdot,\alpha)$ and $w^\delta_\beta=U(\cdot,\beta)$ are $\delta$--geodesics. From (9.1)
\begin{equation}   
\cL_S(u(t_\lambda), v(t_\lambda))\le\cL(U(\cdot,t_\lambda))\le 
(1-\lambda)\int_0^S L\circ\dot w^\delta_\alpha +\lambda\int_0^S L\circ\dot w^\delta_\beta.
\end{equation}
Now $\lim_{\delta\to 0} w^\delta_\alpha=w_\alpha$ and $\lim_{\delta\to 0} w^\delta_\beta=w_\beta$ are the weak geodesics in $\cH^{1\bar 1}$ connecting $u(\alpha), v(\alpha)$, respectively, $u(\beta), v(\beta)$; and, as explained in section 4,
\[
w^\delta_\alpha\to w_\alpha,\quad w^\delta_\beta\to w_\beta,\quad \dot w^\delta_\alpha\to \dot w_\alpha,\quad \dot w^\delta_\beta\to \dot w_\beta
\]
uniformly as $\delta\to 0$. Thus by Lemma 3.4 and Theorem 8.1
\[
\lim_{\delta\to 0} \int^S_0 L\circ\dot w^\delta_\alpha=\int^S_0 L\circ\dot w_\alpha=\cL_S(u(\alpha), v(\alpha)),
\]
and similarly for the other integral in (9.2). Hence letting $\delta\to 0$ in (9.2) gives
\[\cL_S(u(t_\lambda), v(t_\lambda))\le(1-\lambda)\cL_S(u(\alpha), v(\alpha))+\lambda \cL_S(u(\beta), v(\beta)),\]
what was to be proved. 
\end{proof}

\begin{lem}    
If  $v:[a,b]\to B(X)\cap \psho$ is a weak geodesic, then
\[
||v(s)-v(t)||\le \frac{||v(b)-v(a)||}{b-a}|s-t|,\qquad s,t\in[a,b].
\]  
\end{lem}
\begin{proof} This is not new. Let $M=||v(b)-v(a)||$. As $v(\cdot)(x)$ is convex, $\dot v(a)\le(v(b)-v(a))/(b-a)\le M/(b-a)$. Furthermore, 
$u(t)=v(a)-M(t-a)/ (b-a)$ 
is a subgeodesic, $u(a)=v(a)$, $u(b)\le v(b)$. Hence $u(t)\le v(t)$ for all $t$, and
\[
\dot v(a)\ge \lim_{t\to a}\frac{u(t)-v(a)}{t-a}\ge-\frac{M}{b-a}.
\]
Arguing similarly at $b$ we find $||\dot v(a)||,||\dot v(b)||\le M/(b-a)$, and the claim follows, again since $v(\cdot)(x)$ is convex for $x\in X$.
\end{proof}
\begin{lem} 
If $w,w'\in B(X)\cap \psho$ and $T>0$, then $\cL_T(w,w')$ is finite. If $w_j, w'_j\in C(X)\cap \psho$ 
decrease, or converge uniformly, to $w$, resp. $w'$, then
\begin{equation}   
\cL_T(w_j, w'_j)\to \cL_T(w,w') \quad\text{as } j\to\infty.
\end{equation}
\end{lem}
We do not know if (9.3) holds when $w_j,w_j'\in B(X)\cap\psho$.
\begin{proof}
We will prove (9.3) for decreasing sequences $w_j,w_j'$; the case of uniformly convergent sequences can be reduced to 
decreasing sequences in a standard way, as in Lemma 3.4.

Invariance implies that $L$ is constant on the zero section of $T^\infty\cE(\omega)$. Since adding a constant to $L$ will not affect the validity of the lemma, we will assume $L$ vanishes on the zero section. Let us start with (9.3). It suffices to prove
it along a subsequence $j=j_k$.

Assume first that $w_j, w'_j\in\cH$. Let $u:[0, T]\to B(X)\cap\psho$ be piecewise $C^1$ connecting $w$ and $w'$. At the 
price of passing to a subsequence, by 
Lemma 8.2 there are $u_j:[0,T]\to\cH$ piecewise $C^1$ such that $u_j(0)=w_j$, $u_j(T)=w'_j$, and 
$\int_0^TL\circ\dot u_j\to\int_0^TL\circ\dot u$. Therefore
\[
\cL(u)=\lim_{j\to\infty} \cL(u_j)\ge\limsup_{j\to \infty}\cL_T(w_j, w'_j).
\]
Passing to the infimum over paths $u$ connecting $w,w'$, 
\begin{equation}    
\cL_T(w,w')\ge\limsup_{j\to\infty} \cL_T(w_j, w'_j).
\end{equation}
Let $v_j:[0,T]\to\cH^{1\bar 1}$ be the weak geodesics connecting $w_j$ and $w'_j$.

We take a pause in the proof of (9.3) and show how (9.4) implies $\cL_T(w,w')>-\infty$. By Lemma 9.3
 $\|\dot v_j(0)\|$ is a bounded sequence. Since $L$ is equi--Lipschitz on bounded subsets of the fibers (Lemma 3.3), using Lemma 7.3 as well, $\cL_T(w_j,w'_j)=\cL(v_j)=TL(\dot v_j(0))$ is a bounded sequence, and (9.4) implies $\cL_T(w,w')>-\infty$.

We return to the proof of (9.3); we need to estimate $\cL_T(w,w')$ from above. For fixed $\delta>0$ there are infinitely many $k$ with
\begin{equation}   
\liminf_{j\to\infty} \cL_T(w_j,w'_j)\ge\cL_T(w_k, w'_k)-\delta= TL(\dot v_k(0))-\delta.
\end{equation}
If $0<\varepsilon<T/2$, define
\[
v^\varepsilon_k(t)= \begin{cases} tw_k/\varepsilon+(\varepsilon-t)w/\varepsilon & \text{if $0\le t <\varepsilon$}\\
\displaystyle
v_k\Big(t-\varepsilon\frac{T-2t}{T-2\varepsilon}\Big) & \text{if $\varepsilon\le t <T-\varepsilon$}\\
(T-t)w'_k/\varepsilon +(t+\varepsilon-T)w'/\varepsilon & \text{if $T- \varepsilon\le t \le T$.}\end{cases} 
\]
The piecewise $C^1$ paths $v^\varepsilon_k:[0, T]\to B(X)\cap\psho$ connect $w$ and $w'$, hence
\begin{equation}    
\cL_T(w,w')\le \cL(v^\varepsilon_k)=\Big(\int^\varepsilon_0+\int^{T-\varepsilon}_\varepsilon+
\int^T_{T-\varepsilon}\Big) L\circ\dot v^\varepsilon_k.
\end{equation}
The middle integral on the right is
\[
\int^{T-\varepsilon}_\varepsilon L\circ \dot v^\varepsilon_k=(T-2\varepsilon)L\Big(\frac{T\dot v_k(0)}{T-2\varepsilon}\Big).
\]
As we saw, the $\dot v_k(0)$ are uniformly bounded. By the
equi--Lipschitz property of $L$ an $\varepsilon\in (0,1)$ can be chosen so that for all $k$
\begin{equation}     
\int^{T-\varepsilon}_\varepsilon L\circ \dot v^\varepsilon_k\le TL (\dot v_k(0))+\delta.
\end{equation}

When $0\le t\le\varepsilon$, we have $\dot v^\varepsilon_k(t)=(w_k-w)/\varepsilon\in T^\infty_{v^\varepsilon_k (t)}\cE(\omega)$. Again by the equi--Lipschitz property, if $k$ is sufficiently large, $|L(\dot v^\varepsilon_k(t))|< \delta$; and similarly for $T-\varepsilon\le t< T$. Putting this and (9.6), (9.7) together,
\[\cL_T(w,w')\le 3\delta+TL(\dot v_k(0))\]
if $k$ is sufficiently large. Choosing $k$ that also satisfies (9.5) therefore yields
\[\cL_T(w,w')\le 4\delta+\liminf_{j\to\infty} \cL_T(w_j, w'_j).\]
This being true for all $\delta >0$, (9.3) follows in view of (9.4).

So far we dealt with $w_j, w'_j\in\cH$. If $w_j, w'_j\in C(X)\cap\psho$ only, upon adding constants to them we can
arrange that $w_j<w_{j-1}$ and $w_j'<w'_{j-1}$ everywhere. We will express this by saying that $w_j, w_j'$ strictly 
decrease. We construct recursively $z_j>w_j$, $z'_j>w'_j$ in $\cH$ that strictly decrease to $w,w'$ and satisfy 
$|\cL_T(z_j,z_j')-\cL_T(w_j,w_j')|<1/j$ as follows. Suppose we already have $z_{j-1},z'_{j-1}$.
Construct sequences $y_i<z_{j-1}$, $y'_i<z'_{j-1}$ ($i\in\bN$) in $\cH$ that decrease to $w_{j},w'_{j}$. By what we 
have already proved, $|\cL_T(y_i, y'_i)-\cL_T(w_j,w'_j)| < 1/j$ for some $i$, and we let $z_j=y_i$, 
$z'_i=y'_i$ with that $i$. Thus 
\[
\cL_T(w,w')=\lim_j\cL_T(z_j,z'_j)=\lim_j\cL_T(w_j,w'_j),
\]
as claimed. 
\end{proof}

\begin{proof}[Proof of Theorem 9.1]
Assume first that $u,v$ are weak geodesics in $\cH^{1\bar1}$ with endpoints in $\cH$, 
and connect $u(a),u(b)$, respectively, $v(a),v(b)$ by
$\var$--geodesics $u^\var, v^\var$. By Chen's theorem $u^\var\to u$ and $v^\var\to v$ uniformly as $\var\to 0$. Hence
by Lemma 9.4, $\cL_S(u^\var,v^\var)\to \cL_S(u,v)$, and so the latter, as the limit of convex functions (Lemma 9.2) is itself 
convex.

Second, consider general $u,v$. Choose $w_j, w'_j\in\cH$ decreasing to $u(a), u(b)$ and $z_j,z_j'\in\cH$ decreasing to
$v(a),v(b)$. Join $w_j, w'_j$ by weak geodesics $u_j:[a,b]\to\cH^{1\bar 1}$ and $z_j,z'_j$ by weak geodesics 
$v_j:[a,b]\to\cH^{1\bar1}$. By Lemma 8.3 $u_j,v_j$ decrease to $u,v$, hence by Lemma 9.4 the convex functions
$\cL_S(u_j,v_j)$ converge to $\cL_S(u,v)$. It follows that the latter is also convex.
\end{proof}

\section{Two ways to compute action}

One way is by the definition, $\int_a^bL\circ\dot u$, if $u:[a,b]\to B(X)\cap\psho$. The other corresponds to computing
length of a curve in a metric space as the least upper bound of the lengths of inscribed piecewise geodesic curves. The two 
agree in our setting as well.

The Lagrangian $L:T^\infty\cE(\omega)\to\bR$ in this section is strongly continuous, invariant, and convex.

\begin{thm}   
For a piecewise $C^1$ path $u:[a,b]\to B(X)\cap\psho$
\begin{equation}  
\int_a^b L\circ\dot u=\sup\sum_{i=1}^m\cL_{t_i-t_{i-1}}\big(u(t_{i-1}),u(t_i)\big),
\end{equation}
the sup over all partitions $a=t_0<t_1<\dots <t_m=b$.
\end{thm}
We start with an asymptotic formula for $\cL_T(w,z)$, valid as $w,z$ approach each other in two different ways. One will be 
needed in the proof of Theorem 10.1; the other for material in section 11.
\begin{lem}   
Consider a sequence of positive numbers $\tau_k\to 0$ and sequences $w_k,z_k\in B(X)\cap\psho$ converging uniformly
to $w\in B(X)\cap\psho$, $||z_k-w_k||=O(\tau_k)$. Let $\xi\in B(X)$. Suppose that either

(i) $(z_k-w_k)/\tau_k\to\xi$ uniformly or only in capacity; or

(ii) $(z_k-w_k)/\tau_k\to\xi$ $\mu_w$--almost everywhere, and the family of $\mu_{w_k},\mu_{z_k}$, $k\in\bN$, 
is hereditarily tight (Definition 3.5).

Viewing $\xi$ as a vector in $T^\infty_w\cE(\omega)$, we then have
\begin{equation}  
\cL_{\tau_k}(w_k,z_k)/\tau_k \to\ L(\xi) \qquad\text{as}\quad k\to \infty.
\end{equation}
\end{lem}
\begin{proof} 
It will suffice to prove (10.2) along a subsequence. Construct $w_k^j,z_k^j\in\cH$ that decrease
to $w_k$, respectively, $z_k$ as $j\to\infty$. Let $v_k^j:[0,\tau_k]\to\cH^{1\bar1}$ be the weak geodesic connecting 
$w_k^j, z_k^j$, and $u^j_k:[0,\tau_k]\to\cH$ the line segment connecting the two,
\[
u^j_k(t)=\frac{\tau_k-t}{\tau_k}w_k^j+\frac{t}{\tau_k}z_k^j. 
\]
Since $||z_k-w_k||=O(\tau_k)$ as $k\to \infty$, we can arrange that $(z_k^j-w_k^j)/\tau_k$ form a bounded set in $B(X)$. 
This implies by Lemma 9.3 that $\dot v_k^j(t)$ are uniformly bounded. 

We have $v_k^j(t)\le u_k^j(t)$, because evaluated at $x\in X$ the former is a convex function of $t$, the latter  is a 
linear function, and the two agree at $t=0,\tau_k$. Hence
\[
\dot v^j_k(0)\le\dot u^j_k(0), \qquad \dot u^j_k(\tau_k)\le\dot v^j_k(\tau_k), 
\]
and since the right, respectively, left derivatives $\dot v_k^j(0)$, $\dot v_k^j(\tau_k)$ 
are equidistributed, see \cite[Lemma 4.10]{D2},
\begin{equation}
\dot u^j_k(\tau_k)^\star\le\dot v_k^j(0)^\star\le\dot u_k^j(0)^\star.
\end{equation}
As $j\to\infty$, the decreasing sequences $w_k^j$, respectively, $z_k^j$, of $\omega$--plurisubharmonic 
functions converge in capacity  \cite[Proposition 9.11]{GZ2}, and so
\[
\dot u_k^j(0)=\dot u_k^j(\tau_k)=\frac{z_k^j-w_k^j}{\tau_k}\to \frac {z_k-w_k}{\tau_k},\qquad\text{as }j\to\infty,
\]
in capacity.
The latter function is quasicontinuous because $\omega$--plurisubharmonic functions are  \cite[Corollary 9.12]{GZ2}. It 
can be viewed as an element of  $T_{w_k}^\infty\cE(\omega)$ or of $T_{z_k}^\infty\cE(\omega)$. We will write
$f_k,g_k:[0,\mu_0(X)]\to\bR$ for its decreasing rearrangement  as an element of one or the other.
Lemma 3.4 then implies 
\[
\dot u_k^j(0)^\star\to f_k , \quad \dot u_k^j(\tau_k)^\star\to g_k,\qquad\text{as } j\to\infty,
\]
away from a countable subset of $[0,\mu_0(X)]$, and so by (10.3)
\begin{equation} 
g_k\le\liminf_{j\to\infty}\dot v_k^j(0)^\star\le
\limsup_{j\to\infty}\dot v_k^j(0)^\star\le f_k
 \end{equation}
away from a countable set.
In case (i) by Lemma 3.4, in case (ii) by Lemma 3.6 we obtain
\[
\lim_{k\to\infty}f_k=\lim_{k\to\infty}g_k= \xi^\star,
\]
away from a countable set. 

By Egorov's theorem for each $m\in\bN$ there is an $E_m\subset[0,\mu_0(X)]$, whose complement has Lebesgue
measure $<2^{-m}$, and on which the sequences $f_k, g_k$, and for every $k$ the sequences
\[
\inf_{j\ge i}\dot v_k^j(0)^\star,\qquad \sup_{j\ge i}\dot v_k^j(0)^\star,\qquad i=1,2,\dots
\]
converge uniformly. Upon passing to a subsequence we can arrange that for every $m$
\[
\xi^\star-1/m\le g_m\le f_m\le \xi^\star+1/m\qquad\text{on } E_m.
\]
In light of (10.4) for each $m$ there is $i_m$ such that whenever $j>i_m$,
\begin{equation} 
\xi^\star-2/m<\dot v_m^j(0)^\star< \xi^\star+2/m\qquad\text{on } E_m.
\end{equation}
Choose $j>i_m$ so that 
\begin{equation}  
\cL_{\tau_m}(w_m,z_m)/\tau_m\qquad\text{and}\qquad  
\cL_{\tau_m}(w^j_m,z^j_m)/\tau_m=L\big(\dot v^j_m(0)\big)
\end{equation}
are within $1/m$, and set $v_m=v^j_m$.

By the Borel--Cantelli lemma almost every point in $[0,\mu_0(X)]$ is contained in all but finitely many $E_m$.  (10.5)
therefore implies $\lim_m\dot v_m(0)^\star= \xi^\star$ a.e. With a measure preserving $\theta:(X,\mu_0)\to[0,\mu_0(X)]$, see
\cite[Lemma 5.5]{L4}, $\dot v_m(0)^\star\circ\theta\in B(X,\mu_0)$ then tends to $\xi^\star\circ\theta$ a.e., whence 
\[
L\big(\dot v_m(0)\big)=L\big(\dot v_m(0)^\star\circ\theta\big)\to L(\xi^\star\circ\theta)=L(\xi),\qquad m\to\infty,
\]
as $L$ is invariant and strongly continuous. But then
\[
\lim_{m\to\infty} \cL_{\tau_m}(w_m,z_m)/\tau_m=L(\xi)
\]
(cf. (10.6)), as needed.

\end{proof}
\begin{proof}[Proof of Theorem 10.1]
It suffices to prove when $u$ is $C^1$, not only piecewise. Since 
\begin{equation} 
\int_{t_{i-1}}^{t_i}L\circ\dot u\ge\cL_{t_i-t_{i-1}}(u(t_{i-1}),u(t_i)),
\end{equation} 
the left hand side of (10.1) is $\ge$ than the right hand side.
As to the converse, let $M=\max_{[a,b]}|L\circ \dot u|$. 

Given $\var>0$, choose $\delta>0$ so that for any partition $ a=t_0<t_1<\dots<t_m=b$ finer than $\delta$  any
corresponding Riemann sum satisfies
\begin{equation}  
\Big|\int_a^bL\circ\dot u-\sum_{i=1}^mL(\dot u(s_i))(t_i-t_{i-1})\Big|<\var.
\end{equation}
Here $t_{i-1}\le s_i\le t_i$.  It follows from Lemma 10.2 that for every $s\in(a,b)$ there is a $\delta_s\in(0,\delta)$ such that if 
$0<\tau<\delta_s$,
\[
\big|\cL_{\tau}(u(s-\tau/2),u(s+\tau/2))- L(\dot u(s))\tau\big|<\var\tau/(b-a).
\]
Vitali's covering theorem implies that there are a partition $ a=t_0<t_1<\dots<t_m=b$ finer than $\delta$, and 
$I\subset\{1,2,\dots, m\}$ with the following property. Let $s_i=(t_{i-1}+t_i)/2$. If $i\in I$ then $t_i-t_{i-1}<\delta_{s_i}$;
while the total length of the intervals $[t_{i-1},t_i]$ with  $1\le i\le m$ not in $I$ is $<\var/M$.  Write $t_i-t_{i-1}=\tau_i$.
By (10.7) $|\cL_{\tau_i}(u(t_{i-1}),u(t_i))|\le M\tau_i$. We have
\begin{multline*}
\int_a^bL\circ\dot u-\sum_{i=1}^m\cL_{\tau_i}\big(u(t_{i-1}),u(t_i)\big)\\ 
=\Big(\int_a^bL\circ\dot u-\sum_{i=1}^mL(\dot u(s_i))\tau_i\Big)+
\sum_{i=1}^m\big( L(\dot u(s_i))\tau_i-\cL_{\tau_i}(u(t_{i-1}),u(t_i))\big).
\end{multline*}
The first term on the right is $<\var$ according to (10.8). The second is
\[
\sum_{i\in I}+\sum_{i\notin I}\le\sum_{i\in I}\var\tau_i/(b-a)+\sum_{i\notin I}2M\tau_i<\var+2\var.
\]
All added up, $\int_a^bL\circ\dot u-\sum_{i=1}^m\cL_{t_i-t_{i-1}}\big(u(t_{i-1}),u(t_i)\big)< 4\var$, 
and the theorem follows.
\end{proof}

\section{Uniqueness of minimizing paths}   

In this section $L:T^\infty\cE(\omega)\to\bR$ will denote a strongly continuous, invariant, convex Lagrangian. The question 
we will entertain is whether weak geodesics are the unique minimizers of action $\int_a^bL\circ\dot u$ among paths 
connecting fixed $u(a),u(b)$. 

In complete generality uniqueness, of course, fails. For example, if $L$ is positively 
homogeneous, reparametrized weak geodesics will still minimize action. Uniqueness may also fail more drastically 
even with $L$ a Finsler metric. Here is an example. Start with a Lagrangian $\Lambda: T^\infty\cE(\omega)\to\bR$ that
vanishes on constants. For instance, denoting the average of $\xi\in T_u^\infty\cE(\omega)\approx B(X)$ with respect 
to the measure $\mu_u$ by $\langle\xi\rangle_u$,
\[
\Lambda(\xi)=\int_X\big|\xi-\langle\xi\rangle_u\big|\,d\mu_u,\qquad \xi\in T_u^\infty\cE(\omega),
\]
is a possibility. The path $v(t)=t$, $0\le t\le T$, is a geodesic in $\cH$, and in particular minimizes action of the
Lagrangian $L(\xi)=\max\big(\langle\xi\rangle_u,\Lambda(\xi)\big)$.

But take arbitrary nonconstant $w,z \in C^\infty(X)$ and with piecewise $C^1$ functions $f,g:[0,T]\to\bR$ let
\[
u(t)=t+f(t)w+g(t)z, \qquad 0\le t\le T.
\]
We choose $f,g$ so that ($u$ maps into $\cH$ and) $\langle \dot u(t)\rangle_{u(t)}=1$, i.e.,
\begin{multline}  
0=\int_X\big(\dot f(t)w+\dot g(t)z\big)\big(\omega+f(t)dd^c w+g(t) dd^c z\big)^n \\
=\dot f(t)P\big(f(t),g(t)\big)+\dot g(t)Q\big(f(t),g(t)\big),
\end{multline}
where $P,Q$ are polynomials determined by the choice of $w,z$.  
If $f,g$ satisfy 
\[
f(0)=g(0)=0,\qquad\dot f=Q(f,g),\quad \dot g=-P(f,g) \,\text{ on } [0, T/2],
\]
and $f(t)=f(T-t)$, $g(t)=g(T-t)$ for $T/2\le t\le T$, then (11.1) holds. When $T$ is small, this initial value problem
is solvable, and furnishes a path $u$ in $\cH$ connecting $0$ and $T$. Now suppose that $P(0,0)=\int_Xw\omega^n=0$ 
and $Q(0,0)=\int_Xz\omega^n$ is small but nonzero. Since, again for small $T$ and $t\in[0,T]$
\[
\Lambda(\dot u(t))<1=\langle\dot u(t)\rangle_{u(t)}, \qquad\text{and so}\qquad \int_0^TL\circ\dot u=T=\int_0^TL\circ\dot v,
\]
it follows that $u$ also minimizes action; and $u$ is not $v$ reparametrized.

Darvas's $L^1$ metric, $L(\xi)=\int_X|\xi|\,d\mu_u$ for $\xi\in T^\infty_u\cE(\omega)$, supplies another example. 
\cite[Proposition 3.43]{D4} and Theorem 10.1 imply that a $C^1$ path $u:[a,b]\to B(X)\cap \psho$ minimizes action whenever
it is monotone in the sense that $u(t)\le u(s)$ if $t\le s$.

If mere convexity of $L$ does not imply uniqueness of minimizers, strict versions of convexity do, at least in
$\cH^{1\bar1}$:

\begin{thm}  
Suppose $u:[a,b]\to \cH^{1\bar1}$ is piecewise $C^1$ as a map into $B(X)$,  $v:[a,b]\to \cH^{1\bar1}$ is a
weak geodesic
connecting $u(a)$ and $u(b)$, and $\int_a^b L\circ\dot u= \cL_{b-a}(u(a),u(b))$. If $L$ is strictly convex in the sense
that for all $w\in \cH^{1\bar1}$,
$\xi,\eta\in T_w^\infty\cE(\omega)$
\begin{equation}   
L\Big(\frac{\xi+\eta}2\Big)=\frac {L(\xi)+L(\eta)}2\quad\text{implies } \xi=\eta\,\,\mu_w \text{--almost everywhere},
\end{equation}
then $u=v$. If, instead, $L$ satisfies the weaker condition
\begin{equation}   
L\Big(\frac{\xi+\eta}2\Big)=\frac {L(\xi)+L(\eta)}2\quad\text{implies }
\xi=\lambda(\xi+\eta)\,\,\mu_u \text{--almost everywhere }
\end{equation}
with some $\lambda\in[0,1]$, then $u=v\circ\varphi$, where $\varphi:[a,b]\to[a,b]$ is piecewise $C^1$.
\end{thm}

We have not written the minimization condition as $\int_a^bL\circ\dot u=\int_a^bL\circ\dot v$ because we need not
assume a priori that $v$ is $C^1$.
But we do not know if the theorem holds more  generally when $u,v$ map into $B(X)\cap\psho$.---It 
suffices to verify conditions (11.2), (11.3) for $w=0$ only, since Rohlin's theory of Lebesgue spaces 
\cite[Section 2, N\textsuperscript{os}\xspace 3, 4, 7]{R} and the fact that a general $\mu_w$ has no atoms yields a
measure preserving bijection between $(X,\mu_0)$ and $(X,\mu_w)$.

In the proof of the next lemma we will use  Mabuchi length, the Lagrangian $L=M$, 
\begin{equation}  
M(\xi)=\Big(\int_X\xi^2\,d\mu_w\Big)^{1/2},\qquad \xi\in T_w^\infty\cE(\omega).
\end{equation}
The corresponding action $\cM_T(w,w')$, Mabuchi distance, is independent of $T$. We will denote it $d(w,w')$.

\begin{lem}   
Consider a map $u:[a,b]\to\cH^{1\bar1}$. If for some $c\in(a,b)$ the restrictions $u|[a,c]$, $u|[c,b]$ are weak geodesics,
and the left and right derivatives $\partial^-u(c)$, $\partial^+u(c)$, computed pointwise on $X$, agree 
$\mu_{u(c)}$--almost everywhere, then $u$ is a weak geodesic.
\end{lem}
\begin{proof}
It follows from He's work \cite{He} that if $v:[\alpha,\beta]\to \cH^{1\bar1}$ is a weak geodesic, then  the family
$\mu_{v(t)}$, $t\in[\alpha,\beta]$, is hereditarily tight (Definition 3.5).
Indeed, He constructs smooth $v_k:[\alpha,\beta]\to\cH$ that,  as maps 
into $B(X)$, converge to $v$ uniformly  and have $dd^cv_k(t)$ uniformly bounded, $k\in\bN$, $t\in[\alpha,\beta]$, 
see \cite[Theorem 1.3, and proof of
Theorem 1.1]{He}. Therefore $d\mu_{v_k(t)}/d\mu_0$ are uniformly bounded, say, by $A$. Since $\mu_{v(t)}$ is
the weak limit of $\mu_{v_k(t)}$, e.g., by \cite[Theorem 3.18]{GZ2}, for all $\xi\in C(X)$
\[
\Big|\int_X\xi\, d\mu_{v(t)}\Big|=\lim_{k\to\infty}\Big|\int_X\xi\,d\mu_{v_k(t)}\Big|\le A\int_X|\xi|\,d\mu_0.
\]
Hence each $\mu_{v(t)}$ is absolutely continuous with respect to $\mu_0$, with Radon--Nikodym derivative 
$d\mu_{v(t)}/d\mu_0\le A$, and this implies hereditary  tightness.

In particular, the family $d\mu_{u(t)}$, $t\in [a,b]$, is hereditarily tight. 
By Theorem 9.1 $h(t)=d(u(c-t),u(c+t))$ is a convex function of small $t\ge 0$, it vanishes at $0$ and by Lemma 10.2
$h'(0)=2M(\partial^-u(c))=2M(\partial^+u(c))$, cf. (11.4).
Hence $h(t)\ge 2M(\partial^-u(c))t$. At the same time, by Lemma 7.3
\[
d\big(u(c-t),u(c)\big)=d\big(u(c),u(c+t)\big)=M\big(\partial^-u(c)\big)t,
\]
and so
\[
h(t)=d\big(u(c-t),u(c+t)\big)\ge d\big(u(c-t),u(c)\big)+d\big(u(c),u(c+t)\big).
\]
Thus $u|[c-t,c+t]$ is a shortest path for $d$ between $u(c\pm t)$; it is 
of constant speed, too. By \cite[Theorem 1, and the discussion 
following it]{D3} this means that $u|[c-t,c+t]$ is a weak geodesic.  Therefore $u$ itself is a weak geodesic. Indeed,
with the strip $S_{ab}=\{s\in\bC:a<\text{Re}s<b\}$ and projection $\pi:S_{ab}\times X\to X$, we need to check that
\[
U:S_{ab}\times X\ni(s,x)\mapsto u(\text{Re}s)(x)\in\bR
\]
is $\pi^*\omega$--plurisubharmonic and maximal in the sense that $(\pi^*\omega+dd^cU)^{n+1}=0$. Both hold 
because they hold on $S_{ac}\times X$, $S_{cb}\times X$, and on a neighborhood of $\{c+i\bR\}\times X$.
\end{proof}

The key to  Theorem 11.1 is the following characterization of constellations in which a triangle inequality degenerates:
\begin{lem}   
Suppose the Lagrangian $L$ satisfies (11.3), and $w,w',w''\in \cH^{1\bar1}$. If with some $S,T>0$
\[
\cL_S(w',w)+\cL_T(w,w'')=\cL_{S+T}(w',w''),
\]
then $w$ is on the weak geodesic $v:[0,S+T]\to\cH^{1\bar1}$ connecting $w',w''$. If $L$ is strictly convex (condition
(11.2)), then $w=v(S)$.
\end{lem}
\begin{proof} Construct a path $u:[0,S+T]\to B(X)\cap\psho$ whose restrictions $u|[0,S]$, $u|[S,S+T]$
are weak geodesics connecting $w',w$, respectively, $w,w''$. By \cite[Theorem 1.1]{He} $u$ in fact maps into
$\cH^{1\bar1}$. With small $\tau>0$
\begin{align*}
\cL_{S+T}&(w',w'')
\le\cL_{S-\tau}\big(w',u(S-\tau)\big)+\cL_{2\tau}\big(u(S-\tau),u(S+\tau)\big)+\cL_{T-\tau}\big(u(S+\tau), w''\big) \\
\le&\cL_{S-\tau}\big(w',u(S-\tau)\big)+\cL_{\tau}\big(u(S-\tau),w\big)+\cL_\tau\big(w,u(S+\tau)\big)+
     \cL_{T-\tau}\big(u(S+\tau), w''\big) \\
  \le&\cL_S(w',w)+\cL_T(w,w'')=\cL_{S+T}(w',w'').
\end{align*}

Hence $\cL_{2\tau}(u(S-\tau),u(S+\tau))=\cL_{\tau}(u(S-\tau),w)+\cL_\tau(w,u(S+\tau))$. We divide by $2\tau$ and compute the 
limits as 
$\tau\to 0$ using  Lemma 10.2. This is possible, since  the family $\mu_{u(t)}$, $t\in[a,b]$, is hereditarily tight by the
initial observation in the proof of Lemma 11.2. We obtain
\begin{equation}  
L\Big(\frac{\partial^-u(S)+\partial^+u(S)}2\Big)=\frac{L(\partial^-u(S))+L(\partial^+u(S))}2,
\end{equation}
If $L$ is strictly convex, (11.5) implies $\partial^-u(S)=\partial^+u(S)$ $\mu_w$--almost everywhere. By Lemma 11.2
therefore $u$ is a weak geodesic connecting $w',w''$, and so coincides with $v$.

If $L$ satisfies the weaker condition (11.3) only, we can still conclude 
$\partial^-u(S)=\lambda(\partial^-u(S)+\partial^+u(S))$ $\mu_w$--almost everywhere, $0\le\lambda\le1$. If one of 
$\partial^\pm u(S)$ is a.e. $0$, then $w$ is at $0$ Mabuchi distance to $w'$  or $w''$, hence coincides with one of them.
Otherwise $\lambda\neq 0,1$, and $u$ can be linearly reparametized on $[S,S+T]$ to a path $\tilde u$ that satisfies
$\partial^-\tilde u(S)=\partial^+\tilde u(S)$ $\mu_w$--almost everywhere. Lemma 11.2 again implies that $w$ is on
the geodesic $v$.
\end{proof}
\begin{proof}[Proof of Theorem 11.1]
Let $a<s<b, S=s-a$, and $T=b-s$. Then
\begin{multline*} 
\cL_{S+T}\big(u(a),u(b)\big)\le\cL_S\big(u(a),u(s)\big)+\cL_T\big(u(s),u(b)\big) \\
\le\int_a^s L\circ\dot u+\int_s^bL\circ\dot u 
=\int_a^bL\circ\dot u=\cL_{S+T}\big(u(a),u(b)\big),
\end{multline*}
and all inequalities here must be equalities. If $L$ is strictly convex, this implies via Lemma 11.3 that $u(s)=v(s)$, and so
$u=v$. 

If $L$ satisfies (11.3) only, Lemma 11.3 gives $u(s)=v(\varphi(s))$ with some function $\varphi:[a,b]\to[a,b]$.
If $v$ is constant, this again means $u=v$. Otherwise, on the one
hand, Mabuchi distance along the weak geodesic $v$ is $d(v(a),v(t))=c(t-a)$, with $c\neq 0$. On the other, Mabuchi 
distance along $u$ is a piecewise $C^1$ function of $s$
\[
d(u(a),u(s))=\int_a^s\Big(\int_X\dot u(t)^2\,d\mu_{u(t)}\Big)^{1/2}\,dt.
\]
It follows that  $d\big(v(\varphi(a)),v(\varphi(s))\big)=c(\varphi(s)-\varphi(a))$ is a piecewise $C^1$ function of $s$, and so is 
$\varphi(s)$.

\end{proof}

\end{document}